\documentclass[final]{siamart1116}
\makeatletter
\def\input@path{{../}{./}}
\makeatother

\usepackage{amsmath,amssymb,amsfonts,bm,mathrsfs}

\newtheorem{remark}[theorem]{Remark}

\usepackage{graphicx}
\graphicspath{{./figures/}}
\usepackage{caption,subcaption}

\usepackage{xcolor}
\usepackage{hyperref}
\hypersetup{
	plainpages=false,
    colorlinks,
    linktocpage=true,   % makes the page number as hyperlink in table of content
    linkcolor={red!50!black},
    citecolor={blue!50!black},
    urlcolor={blue!80!black}
} 
\usepackage[hyperpageref]{backref}

\newcommand{\norm}[1]{\left\Vert#1\right\Vert}

\makeatletter
\newcommand{\tnorm}{\@ifstar\@tnorms\@tnorm}
\newcommand{\@tnorms}[1]{%
\left|\mkern-2.5mu\left|\mkern-2.5mu\left|
#1
\right|\mkern-2.5mu\right|\mkern-2.5mu\right|
}
\newcommand{\@tnorm}[2][]{%
\mathopen{#1|\mkern-2.5mu #1|\mkern-2.5mu #1|}
#2
\mathclose{#1|\mkern-2.5mu #1|\mkern-2.5mu #1|}
}
\makeatother

\newcommand{\jump}[2]{\lbrack\hspace{-1.5pt}\lbrack {#1} 
\rbrack\hspace{-1.5pt}\rbrack_{\raisebox{-2pt}{\scriptsize$#2$}} }

\newcommand{\dd}{\,{\rm d}}
\newcommand{\bs}{\boldsymbol}

\newcommand{\SC}[1]{\textcolor{red}{#1}}
\newcommand{\revision}[1]{\textcolor{black}{#1}}

\newcounter{mnote}
\let\oldmarginpar\marginpar
\renewcommand\marginpar[1]{\-\oldmarginpar[\raggedleft\footnotesize #1]%
{\raggedright\footnotesize #1}}

\numberwithin{theorem}{section}
\numberwithin{equation}{section}

\headers{Error Estimates of The Linear Nonconforming VEM}{S. Cao and L. Chen}

\title{{Anisotropic 
Error Estimates of The Linear Nonconforming Virtual Element Methods
}
\thanks{This paper is based upon work supported by the National Science 
Foundation under Grant No. DMS-1418934.}}

\author{
  Shuhao Cao\thanks{Department of Mathematics, University of California Irvine, 
  Irvine, CA 92697 (\email{scao@math.uci.edu}, \email{chenlong@math.uci.edu}).}
  \and
  Long Chen\footnotemark[2]
}

\usepackage{amsopn}

\ifpdf
\hypersetup{
  pdftitle={Anisotropic Error Estimates of The Linear Nonconforming VEM},
  pdfauthor={S. Cao and L. Chen}
}
\fi

\date{\today}

\begin{document}

\maketitle

\begin{abstract}
A refined a priori error analysis of the lowest order (linear) nonconforming virtual 
element method (VEM) for approximating a model Poisson problem is developed in both 
two and three dimensions. A set of new geometric assumptions is proposed for the 
shape regularity of polytopal meshes. A new error equation for the lowest order 
(linear) nonconforming VEM is derived for any choice of stabilization, and a new 
stabilization using a projection on an extended element patch is introduced for the 
error analysis on anisotropic elements.
\end{abstract}

% REQUIRED

% REQUIRED

\section{Introduction}
In this paper, we develop a modified nonconforming virtual 
element method (VEM), together with a new way to perform the a priori error 
analysis for a model Poisson equation. The new analysis incorporates several new 
geometry assumptions on polytopal partitions in both two and three dimensions.

To approximate multiphysics problems involving complex geometrical features using 
finite element method (FEM) in 2-D and 3-D, how to encode these geometric 
information 
into the discretization is a challenge. To a specific problem's interest, common 
practices include either to generate a body/interface-fitted mesh by cutting a 
shape-regular background mesh, or to build cut-aware approximation 
spaces/variational forms (stencils) on the unfitted background mesh. Some notable 
methods utilizing the latter idea include eXtended FEM (e.g., 
see~\cite{Dolbow;Belytschko:1999finite,Sukumar;Moes;Moran;Belytschko:2000Extended}),
fictitious domain FEM~\cite{Haslinger;Renard:2009fictitious}, cut FEM 
\cite{Burman;Claus;Hansbo;Larson:2015CutFEM:}, and immersed FEM 
\cite{Li;Lin;Wu:2003Cartesian}.

One resolution combining the advantages of both approaches in 3-D was  
proposed in~\cite{Chen;Wei;Wen:2017interface-fitted} by using polyhedral meshes 
rather than the tetrahedral ones. It avoids manually tweaking problematic 
tetrahedra like slivers with four vertices nearly coplanar, which is usually an 
unavoidable problem in generating body-fitted mesh from a background mesh, 
especially when the mesh is fine.

Since arbitrary-shaped polygons or polyhedra are now introduced into the partition, 
it requires that the underlying finite element methods can handle these kinds 
of general meshes. There are several classes of modifications of classical numerical 
methods to work on the polytopal meshes including mimetic finite difference 
(MFD)~\cite{Brezzi;Buffa;Lipnikov:2009Mimetic,Beirao-da-Veiga;Lipnikov;Manzini:2014mimetic},
generalized barycentric 
coordinates~\cite{Gillette;Rand;Bajaj:2012Estimates}, compatible discrete operator 
scheme~\cite{Bonelle;Ern:2014Analysis}, composite/agglomerated discontinuous 
Galerkin finite element 
methods (DGFEM)~\cite{Antonietti;Cangiani;Collis;Dong:2016Review}, 
hybridizable discontinuous Galerkin (HDG) 
methods~\cite{Castillo;Cockburn;Perugia;Schotzau:2000priori,Cockburn;Di-Pietro;Ern:2016Bridging},
 hybrid high-order (HHO) methods
\cite{Di-Pietro;Ern;Lemaire:2014arbitrary-order,Di-Pietro;Ern:2015hybrid}, weak 
Galerkin (WG) methods~\cite{Wang;Ye:2014Galerkin,Mu;Wang;Ye:2015Galerkin}, 
discontinuous Petrov-Galerkin 
(PolyDPG) methods~\cite{Astaneh;Fuentes;Mora;Demkowicz:2018High-order}, etc. 
Among them, the virtual element method 
(VEM) introduced in~\cite{Beirao-da-Veiga;Brezzi;Cangiani;Manzini:2012Principles} 
proposed a universal framework for constructing approximation spaces and proving 
optimal order convergence on polytopal elements. Until now VEMs for elliptic 
problems have been developed with elaborated details (e.g., 
see~\cite{Ahmad;Alsaedi;Brezzi;Marini:2013Equivalent,
Beirao-da-Veiga;Brezzi;Marini;Russo:2014hitchhikers,Brezzi;Falk;Marini:2014principles,
Ayuso-de-Dios;Lipnikov;Manzini:2016nonconforming,Cangiani;Manzini;Sutton:2016Conforming,
Beirao-da-Veiga;Brezzi;Marini;Russo:2016Virtual}).

The nonconforming finite element method for elliptic problems, better known as 
Crouzeix-Raviart element, was introduced in~\cite{Crouzeix;Raviart:1973Conforming}.  
It is nonconforming in the sense that the approximation polynomial space is not a 
subspace of the underlying Sobolev space corresponding to the continuous weak 
formulation. Its VEM counterpart was constructed 
in~\cite{Ayuso-de-Dios;Lipnikov;Manzini:2016nonconforming}. The degrees of freedom 
(DoFs) of a nonconforming VEM function on an element $K$ are the 
natural dual to this function's values according to a Neumann boundary 
value problem on $K$, which are induced by the integral by parts. 
When a locally constructed stabilization term satisfies the patch test, the 
convergence in 
broken $H^1$-seminorm is obtained through a systematized approach by showing the 
norm equivalence for the VEM 
functions between the broken Sobolev norm and the norm induced by the bilinear 
form~\cite{Ayuso-de-Dios;Lipnikov;Manzini:2016nonconforming}.

Establishing the norm equivalence above requires geometric constraints on the shape 
regularity of the mesh. Almost all VEM error analyses to date are performed on 
star-shaped elements, and the mostly used assumptions are 
(1) every element $K$ and every face $F\subset \partial K$ are star-shaped with the 
chunkiness parameter uniformly bounded above; 
(2) no short edge/small face, i.e., $h_F \eqsim h_K$ for every face $F\subset 
\partial K$.  In the former condition, the so-called 
chunkiness parameter of a star-shaped domain $E$ is the ratio of the diameter of $E$ 
over 
the radius of the largest inscribed ball with respect to which $E$ is star-shaped, 
which may become unbounded for anisotropic elements or anisotropic faces in 3-D 
star-shaped elements.

\revision{Recently, some refined VEM error analyses 
(see~\cite{Beirao-da-Veiga;Lovadina;Russo:2017Stability,Brenner;Sung:2018Virtual}) 
have removed the ``no short edge'' assumption in the 2-D conforming VEM by 
introducing a new tangential derivative-type stabilization first proposed in 
\cite{Wriggers;Rust;Reddy:2016virtual}. 
In the 3-D case~\cite{Brenner;Sung:2018Virtual}, the removal of the ``no small 
face'' comes at a price in that the convergence constant depends on the log 
of the ratio of the longest edge and the shortest edge on a face of a polyhedral 
element, which also appears in the 2-D analysis using the traditional DoF-type 
stabilization.} This factor seems non-removable due to the norm equivalence being 
used in these approaches, and it 
excludes anisotropic elements and/or \revision{isotropic} elements with anisotropic 
faces with high aspect ratios in 3-D (e.g. see Figure~\ref{fig:anisoelem}). However, 
in a variety of 
numerical tests, some of which even use the traditional stabilization 
that is suboptimal in theory, VEM performs robustly 
regardless of these seemingly artificial geometric constraints in situations 
like random-control-points Voronoi meshes, irregular 
concave meshes, a polygon degenerating to a line, interface/crack-fitted meshes (see 
\cite{Beirao-da-Veiga;Brezzi;Marini;Russo:2016Virtual,
Beirao-da-Veiga;Dassi;Russo:2017High-order,
Benedetto;Berrone;Pieraccini;Scialo:2014virtual,
Berrone;Borio:2017Orthogonal,Chen;Wei;Wen:2017interface-fitted,
Dassi;Mascotto:2018Exploring,Mascotto:2018Ill-conditioning}). Especially, 
anisotropic elements and/or elements with anisotropic faces pose no 
bottlenecks to the convergence of VEM numerically.

In an effort to partially explain the robustness of VEM regarding the shape 
regularity of the mesh, in~\cite{Cao;Chen:2018Anisotropic}, an a priori error 
analysis for the lowest order conforming VEM is conducted based on a mesh dependent 
norm $\tnorm{\cdot}$ induced by the bilinear form, which is weaker than 
$H^1$-seminorm. The main instrument is an error equation similar to the ones used in 
the error analysis in Discontinuous Galerkin (DG)-type methods, thus bypassing the 
norm equivalence. In this way, less geometric constraints are required than the 
error analysis using the norm 
equivalence. However, results in~\cite{Cao;Chen:2018Anisotropic} are restricted to 
2-D, and the anisotropic error analysis is restricted to a special class of elements 
cut from a shape regular mesh. In particular, long edges in an anisotropic element 
are required to be paired in order to control the interpolation error in different 
directions. A precise quantitative characterization of such anisotropic meshes, on 
which the analysis can be applied, is not explicitly given 
in~\cite{Cao;Chen:2018Anisotropic}. 

In this paper, we follow this approach, and derive an error equation for the 
lowest order nonconforming VEM. Thanks to the natural definition of DoFs, the 
nonconforming interpolation defined using DoFs brings no 
error into the error estimate in the sense that $\tnorm{u-u_I} = 0$, compared with 
the error estimates of the conforming interpolant being proved using an intricate 
edge-pairing technique in~\cite{Cao;Chen:2018Anisotropic}. As a result, under 
geometric conditions introduced in~\cite{Cao;Chen:2018Anisotropic}, the
anisotropic error analysis can be extended to the lowest order nonconforming VEM in 
both two and three dimensions. 

The findings in this paper strengthens our 
opinion: one of the reasons why VEM is immune to badly shaped elements is that the 
approximation to the gradient of an $H^1$-function is handled by the projection of 
the gradient of a VEM function, not the exact gradient of it. On the other hand, the 
flexibility of the VEM framework allows 
us to modify the stabilization in two ways from the 
one used in~\cite{Ayuso-de-Dios;Lipnikov;Manzini:2016nonconforming} tailored for the 
anisotropic elements: 
(1) the weight is changed from the size of each face, respectively, to the diameter 
of an extended element patch; (2) the stabilization stencil enlarges to this 
extended element patch, and its form remains the same with the 
original DoF-type integral, in which the penalization computes now the difference of 
the VEM functions and their projections onto this extended element patch, not the 
underlying anisotropic element.
In this way, the anisotropic elements can be integrated into the analysis naturally 
using the tools improved from the results 
in~\cite{Wang;Ye:2014Galerkin,Li;Melenk;Wohlmuth;Zou:2010Optimal}, and an 
optimal order convergence can be proved in this mesh dependent norm $\tnorm{\cdot}$. 
\revision{Our stabilization has the same spirit as the so-called ghost penalty 
method introduced in~\cite{Burman:2010penalty} for fictitious domain methods.}

When extending the geometric conditions in~\cite{Cao;Chen:2018Anisotropic} from 2-D 
to 3-D in Section \ref{sec:geometry}, some commonly used tools in finite element 
analysis, including various trace 
inequalities and Poincar\'e inequalities, for simplexes are revisited for polyhedron 
elements. The conditions these inequalities hold serve as a motivation to propose a 
set of constraints as minimal as possible on the shapes of elements. In this regard, 
Assumptions \hyperref[asp:B]{{\bf B--C}} are proposed with more 
local geometric conditions than the star-shaped condition, which in our opinion is a 
more ``global''-oriented condition for a certain element. Moreover, the hourglass 
condition in Assumption \hyperref[asp:C]{{\bf C}} allows the approximation on 
``nice'' hourglass-shaped elements, which further relaxes a constraint in the 
conforming case in~\cite{Cao;Chen:2018Anisotropic} in which vertices have to be 
artificially added to make hourglass-shaped elements isotropic. 

As mentioned earlier, the way to deal with an anisotropic element is to assume one 
can embed this element into an isotropic extended element patch in Assumption 
\hyperref[asp:D]{{\bf D}}.  However the current analysis forbids the existence of 
a cube/square being cut into thin slabs, in which the number of cuts $\to \infty$ 
when $h\to 0$. \revision{From the standpoint of the implementation, the total number 
of the anisotropic elements cannot make up a significant portion of all elements in 
practice, as the enlarged stencil for the modified stabilization makes the stiffness 
matrix denser.}

This paper is organized as follows: In Section \ref{sec:ncvem}, the 
linear nonconforming VEM together with our modification are introduced. Section 
\ref{sec:geometry} discusses the aforementioned set of new geometric assumptions in 
2-D and 3-D. In Section \ref{sec:error}, we derive a new error equation and an 
\textit{a priori} error bound for the linear nonconforming VEM. Lastly in Section 
\ref{sec:remark}, we study how to alter the assembling procedure in the 
implementation. 

For convenience, $x \lesssim y$ and $z \gtrsim w$ are used to represent $x \leq c_1 
y$ and $z \geq c_2 w$ respectively, and $a \eqsim b$ means  $a \lesssim b$ and $a 
\gtrsim b$. The constants involved are independent of the mesh size $h$. When there 
exists certain dependence of these relations to certain geometric properties, then 
such dependence shall be stated explicitly.

\section{Nonconforming Virtual Element Methods}
\label{sec:ncvem}
In this section we shall introduce the linear nonconforming virtual element space 
and corresponding discretization of a model Poisson equation. In order to deal with 
anisotropic elements, we shall propose a new stabilization term.

Let $\Omega$ be a bounded polytopal domain in $\mathbb{R}^d$ $(d=2,3)$, consider the 
model Poisson equation in the weak form with data $f\in 
L^2(\Omega)$: to find $u\in H_0^1(\Omega)$ such that
\begin{equation}
\label{eq:model-weak}
a(u,v) :=(\nabla u, \nabla v) = (f, v)\quad \forall v\in H_0^1(\Omega).
\end{equation}

Provided with the mesh satisfying the assumptions to be discussed in Section 
\ref{sec:geometry}, the goal of this subsection is to build the following 
discretization using a bilinear form $a_h(\cdot,\cdot)$ in a VEM approximation 
space $V_h$ on a given mesh $\mathcal{T}_h$, which approximates the original 
bilinear form $a(\cdot,\cdot)$:
\begin{equation}
\label{eq:dis}
\text{To find }\; u_h\in V_h, \; \text{ such that} \;\; a_h( u_h, v_h) = 
\langle f, v_h\rangle \quad \forall v_h \in 
V_h,
\end{equation}
where $\langle f, v_h\rangle \approx (f, v_h)$.
 
\subsection{Notation}
Throughout the paper the standard notation $(\cdot,\cdot)_D$ are used to denote the 
$L^2$-inner product on a domain/hyperplane $D$, and the subscript is omitted when 
$D = \Omega$. For every geometrical object $D$ and for every 
integer $k \geq 0$, $\mathbb{P}_k(D)$ denotes the set of polynomials of degree 
$\leq k$ on 
$D$. The average of an 
$L^1$-integrable function or vector field $v$ over $D$, endowed with the usual 
Lebesgue measure, is denoted by: $\overline{v}^{D} = |D|^{-1}\int_{D} 
v$, where $|D| = \operatorname{meas}(D)$.

To approximate problem \eqref{eq:model-weak}, firstly $\Omega$ is partitioned into 
a polytopal mesh $\mathcal{T}_h$, each polytopal element is either a simple polygon 
($d=2$) or a simple polyhedron ($d=3$). The set 
of the elements contained in a subset $D\subset \Omega$ is denoted by 
$\mathcal{T}_h(D):=\{K\in \mathcal{T}_h: K\subset \bar D \}$.
$h := \max\limits_{K\in \mathcal{T}_h} h_K$ 
stands for the mesh size, with $h_D:= \operatorname{diam} D$ for any bounded 
geometric object $D$. Denote $\operatorname{conv}(D)$ be the convex 
hull of $D$. The term ``face'' $F$ is usually used to refer to the 
$(d-1)$-flat face of a $d$-dimensional polytope in this partition ($d=2,3$). For 
$d=2$ case, a face refers to an edge unless being otherwise specifically stated. The 
set of all the faces in $\mathcal T_h$ is denoted by $\mathcal{F}_h$. The 
set of the face $F$ on the boundary of an element $K$ is denoted by 
$\mathcal{F}_h(K)$, and $n_K:= |\mathcal{F}_h(K)|$ is the number of faces on the 
boundary of $K$. More generally $\mathcal 
F_h(D):=\{F\in \mathcal{F}_h: F\subset \bar{D} \}$ denotes faces restricted to a 
bounded domain $D$. With the help from the context, $\bm{n}_F$ denotes the outward 
unit normal vector of face $F$ with 
respect to the element $K$. An interior face $F\in \mathcal{F}_h$ is shared by two 
elements $K^{\pm}$. For any function $v$, define the jump 
of $v$ as $\jump{v}{F} = v^- - v^+$ on $F$, where $v^{\pm} = 
\lim\limits_{\epsilon\to 0}v(\bm{x}-\epsilon 
\bm{n}_{F}^{\pm})$, and $\bm{n}_{F}^{\pm}$ represents the outward unit normal vector 
respect to $K^{\pm}$. For a boundary face $F\subset \partial \Omega$, $\jump{v}{F} 
:= v|_F$. 

For a bounded Lipschitz domain $D$, $\norm{\cdot}_{0,D}$ denotes the $L^2$-norm, and 
$|\cdot|_{s,D}$ is the $H^s(D)$-seminorm. Again when $D = \Omega$ being the whole 
domain, the subscript $\Omega$ will be omitted. 

\subsection{Nonconforming VEM spaces}
The lowest order, i.e., the 
linear nonconforming VEM~\cite{Ayuso-de-Dios;Lipnikov;Manzini:2016nonconforming}, is 
the main focus of this article. The 
linear nonconforming VEM has rich enough content to demonstrate anisotropic meshes' 
local impact on the a priori error analysis, and yet elegantly simple enough to be 
understood without many technicalities. Our main goal is to develop the tools 
for the linear nonconforming VEM to improve the anisotropic error analysis for the 
VEM.

The lowest order nonconforming virtual element space $V_h$, restricted 
on an element $K$, can be defined as 
follows~\cite{Ayuso-de-Dios;Lipnikov;Manzini:2016nonconforming}:
\begin{equation}
\label{eq:space-local}
V_h(K):= \bigl\{v\in H^1(K): \quad 
\Delta v = 0\text{ in } K,\, \nabla v\cdot \bm{n}\big\vert_F\in 
\mathbb{P}_{0}(F),\forall F\in \mathcal{F}_h(K) 
\bigr\}.
\end{equation}
The degrees of freedom (DoFs) for the local space $V_h(K)$ is the average  
of $v_h\in V_h(K)$ on every face $F\in \mathcal{F}_h(K)$:
\begin{equation}
\label{eq:dofs-face}
\chi_F(v_h)  = \frac{1}{|F|}\int_{F} v_h \,\dd S.
\end{equation}
Denote by this set of DoFs by $\mathcal N(K) = \{ \chi_F, F\in \mathcal{F}_h(K) \}$ 
with 
cardinality $|\mathcal N(K)| = n_K$, then one can easily verify that $(K, V_h(K), 
\mathcal N(K))$ forms a finite element triple in the sense of 
Chapter 2.3 in~\cite{Ciarlet:1978Finite} 
(see~\cite{Ayuso-de-Dios;Lipnikov;Manzini:2016nonconforming}). 

The global nonconforming VEM space $V_h$ can be then defined as:
\begin{equation}
\label{eq:space-global}
V_h = \bigl\{ v\in L^2(\Omega): \;
v\big\vert_{K} \in V_h(K), \;\forall K\in \mathcal{T}_h, \;\int_F \jump{v}{F} \dd S 
= 0, \; \forall 
F\in \mathcal{F}_h \bigr\}.
\end{equation}
The canonical interpolation 
$v_I\big\vert_K\in V_h(K)$ in the nonconforming VEM local space of $v\in H^1(K)$ is 
defined using the DoFs:
\begin{equation}
\label{eq:intp-loc}
\chi_F(v) = \chi_F(v_I),   \quad \forall F\in \mathcal{F}_h(K),
\end{equation}
and the canonical interpolation $v_I\in V_h$ is then defined using the global DoFs: 
\begin{equation}
\label{eq:intp}
\chi_F(v) = \chi_F(v_I),  \quad \forall F\in \mathcal{F}_h.
\end{equation}

\subsection{Local projections}
The shape functions in $V_h(K)$ do not have to be formed explicitly in assembling 
the stiffness matrix. Based on the construction in 
\eqref{eq:space-local}, locally on an 
element $K$, a certain shape function is the solution to a Neumann boundary value 
problem, the exact pointwise value of which is unknown. Instead, for $u_h,v_h\in 
V_h(K)$, some computable quantities based on the DoFs of $u_h$ and $v_h$ are used to 
compute $a_h(u_h,v_h)$, which approximates the original 
continuous bilinear form $a(u_h,v_h)$. We now explore what quantities can be 
computed explicitly using DoFs.

First of all, the $L^2$-projection $Q_F: v\mapsto Q_F v \in \mathbb{P}_{0}(F)$ for 
any $v\in L^1(F)$ to piecewise constant space on a face $F$ is defined as:
\begin{equation}
\label{eq:pr-L2}
\bigl(v - Q_F v, q\bigr)_F = 0, \quad \forall q \in \mathbb{P}_{0}(F).
\end{equation}
For a VEM function $v_h\in V_h(K)$, this projection can be directly derived from the 
DoFs \eqref{eq:dofs-face}, since $Q_F(v_h) =\chi_F (v_h)$ by definition. 
In contrast, the $L^2$-projection $Q_K: L^1(K) \to \mathbb P_0(K)$:
\begin{equation}
\label{eq:pr-L2K}
\bigl(v - Q_K v, q\bigr)_K = 0, \quad \forall q \in \mathbb{P}_{0}(K).
\end{equation}
 is not computable for $v_h\in V_h(K)$ by using only the DoFs of $v_h$.

On an element $K$, we can also compute an elliptic projection to the linear 
polynomial space: for any $v \in 
H^1(K)$, $\Pi_K v \in \mathbb{P}_1(K)$ satisfies
\begin{equation}
\label{eq:pr}
\left(\nabla \Pi_K v , \nabla q\right)_K = \left(\nabla v, \nabla 
q\right)_K, \quad 
\text{for all } q \in \mathbb{P}_1(K).
\end{equation}
By choosing $q = x_i, i=1,\ldots, d,$ one can easily verify $\nabla \Pi_K v = 
Q_K(\nabla u)$. Namely $\nabla \Pi_K v$ is the best constant approximation of 
$\nabla u$ in $K$.

As $H^1$-semi-inner product is used in \eqref{eq:pr}, $\Pi_K v$ is unique up to a 
constant. The constant kernel will be eliminated by 
the following constraint:
\begin{equation}
\label{eq:pr-constraint}
\begin{aligned}
\int_{\partial K} \Pi_K v \dd S= \int_{\partial K} v \dd S = \sum_{F\in 
\mathcal{F}_h(K)}\,\chi_F(v) |F|.
\end{aligned}
\end{equation}

Using integration by part, and the fact $\Delta q = 
0$, $\nabla q$ being constant for $q\in \mathbb P_1(K)$, the right hand side of 
\eqref{eq:pr} can be written as
\begin{equation}
\label{eq:rhspr}
\left(\nabla v, \nabla q\right)_K = (v, \nabla q\cdot \bs n)_{\partial K} = 
\sum_{F\in \mathcal{F}_h(K)}\nabla q\cdot \bs n_F\,\chi_F(v) |F|.
\end{equation}
Thus for a VEM function $v_h\in V_h(K)$, $\Pi_Kv_h$ can be computed by the DoFs of 
$v_h$.

The following lemma shows that $\Pi_K$ mapping depends only on DoFs. In this 
regards, the elliptic projection 
$\Pi_K$ works in a more natural way for nonconforming VEM local space, 
thanks to the choice of DoFs being the natural dual from the integration by parts.
\begin{lemma}
\label{lemma:pr-id}
For $v,w\in H^1(K)$, where $K\in \mathcal{T}_h$, if for all $F\in \mathcal{F}_h(K)$,
$\chi_F(v) = \chi_F (w)$, then $\Pi_Kv= \Pi_Kw$.
\end{lemma}
\begin{proof}
This is a direct consequence of definition of $\Pi_K$ in view of 
\eqref{eq:pr-constraint}-\eqref{eq:rhspr}. 
\end{proof}

To incorporate the possibility of the 
anisotropic analysis, we shall define an extended element patch containing $K$
$$\omega_K:={\textstyle \bigcup}_{\alpha\in \Lambda} K_{\alpha}$$ where  
$\Lambda = 
\Lambda(K)$ is an index set related to $K$ such that $K \subseteq \omega_K$, 
$K_{\alpha}\in \mathcal{T}_h$ for all $\alpha\in 
\Lambda$, and $\omega_K$ is 
isotropic in the sense of Assumption \hyperref[asp:A]{\bf A--B--C} that 
shall be elaborated in Section \ref{sec:geometry}; for example, see Figure 
\ref{fig:cutelem}. When $K$ 
itself is isotropic, $\omega_K = K$. 
\begin{figure}[h]
  \centering
  \begin{subfigure}[b]{0.3\linewidth}
    \centering\includegraphics[width=90pt]{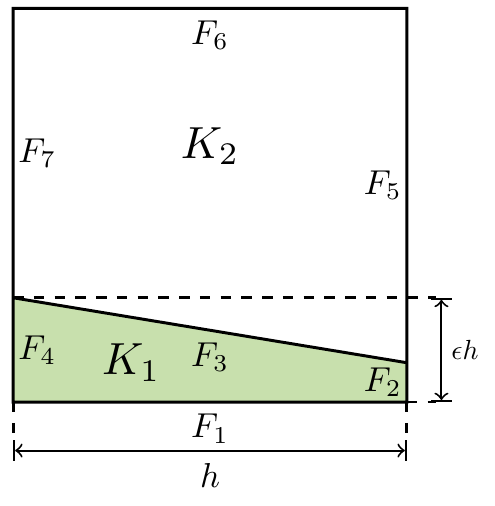}
    \caption{\label{fig:cutelem}}
  \end{subfigure}%
\quad 
\begin{subfigure}[b]{0.33\linewidth}
      \centering\includegraphics[width=110pt]{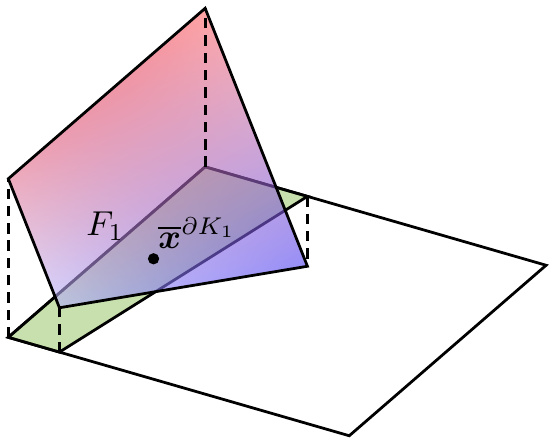}
      \caption{\label{fig:projK}}
    \end{subfigure}
  \begin{subfigure}[b]{0.33\linewidth}
    \centering
    \centering\includegraphics[width=110pt]{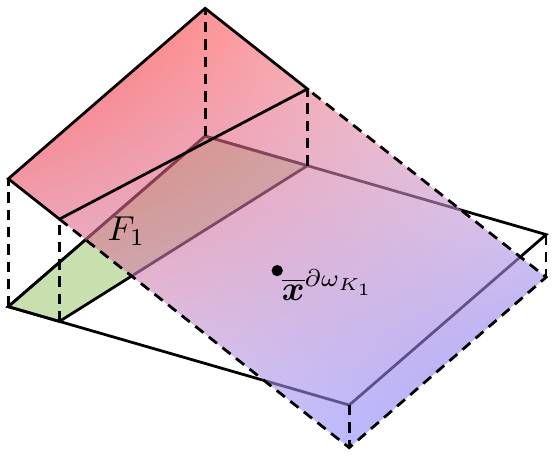}
    \caption{\label{fig:projomega}}
  \end{subfigure}
\vspace{-10pt}
\caption{An illustration of the extended element patch and the elliptic projections 
on it. As $h\to 0$, $\epsilon\to 0$. (\subref{fig:cutelem}) $K_1$ is anisotropic 
and $\omega_{K_1} = K_1\cup K_2$ is isotropic. (\subref{fig:projK}) 
$\Pi_{K_1}\phi_{F_1}$ in \eqref{eq:pr} has sharp 
gradient. (\subref{fig:projomega}) $\Pi_{\omega_{K_1}}\phi_{F_1}$ in 
\eqref{eq:pr-pw} has smoother 
gradient over $\omega_{K_1}$ and is used only in the stabilization term on $\partial 
K_1$, not on $\partial K_2$.}
\label{fig:proj-example}
\end{figure}

We define a discrete $H^1$-type projection on $\omega_K$ as follows: given 
a $v_h \in V_h$
\begin{equation}
\label{eq:pr-pw}
\left(\nabla  \Pi_{\omega_K}v_h, \nabla q\right)_{\omega_K}
= \sum_{K\in \mathcal{T}_h(\omega_K)} \left( \nabla v_h, \nabla q\right)_{K}, 
\quad \forall  q\in \mathbb{P}_1(\omega_K).
\end{equation}
Notice here by the continuity condition in \eqref{eq:space-global}, it is 
straightforward to verify that using 
the integration by parts, for $q\in \mathbb{P}_1(\omega_K)$, on any $F\in 
\mathcal{F}_{h}(\omega_K)$, $\nabla q\cdot \bm{n}\big\vert_F \in \mathbb{P}_0(F)$, 
and $\Delta q = 0$, we have
\begin{equation*}
\begin{aligned}
& \sum_{K\in \mathcal{T}_h(\omega_K)} \left( \nabla v_h, \nabla q\right)_{K}
=\sum_{K\in \mathcal{T}_h(\omega_K)} \left( v_h, \nabla q\cdot 
\bm{n}\right)_{\partial K}
\\
=\! & \sum_{F\in \mathcal{F}_{h}(\partial \omega_K)}
\left( v_h, \nabla q\cdot \bm{n}\right)_{F}
+ \sum_{\substack{F\in \mathcal{F}_{h}(\omega_K),\\F\not\subset \partial 
\omega_K}}\left( \jump{v_h}{F}, \nabla q\cdot \bm{n}\right)_{F}
=\sum_{F\in \mathcal{F}_{h}(\partial \omega_K)}  \left( v_h, \nabla q\cdot 
\bm{n}\right)_{F}, 
\end{aligned}
\end{equation*}
of which the right hand side can be evaluated using the DoFs of $v_h$ similar to 
\eqref{eq:rhspr}.
When $K\neq \omega_K$, the constraint for $\Pi_{\omega_K}$, as well as for 
$\Pi_{K}$ (cf. \eqref{eq:pr-constraint}), is chosen as the average on the 
boundary of $\omega_K$: 
for $v_h \in V_h$
\begin{equation}
\label{eq:pr-constraint-aniso}
\int_{\partial \omega_K} \Pi_{K} v_h \dd S  = \int_{\partial \omega_K} 
\Pi_{\omega_K} v_h \dd S = \int_{\partial \omega_K} v_h \,{\rm d} S,
\end{equation}
which are both computable using DoFs of $v_h$. 

In summary, although we do not have access to the pointwise value of $v_h\in 
V_h(K)$, we can find its average on each face and a linear polynomial $\Pi_K v_h$ 
inside $K$, whose gradient is the best piecewise constant approximation of the 
element-wise gradient of $v_h$. When needed, we can compute another linear 
polynomial 
$\Pi_{\omega_K} v_h$ on an extended patch $\omega_K$ (e.g., see Figure 
\ref{fig:projomega}), \revision{the implementation 
details of which we refer the reader to Section \ref{sec:remark}.}

\subsection{Discretization}
As the $H^1$-projection, $\bigl(\nabla \Pi_K u_h, \nabla \Pi_K v_h\bigr)_K$ is a 
good approximation of $(\nabla u_h, \nabla v_h)_K$. However, $\bigl(\nabla \Pi_K 
u_h, \nabla \Pi_K v_h\bigr)_K$ alone will not lead to a stable method as 
$|\ker(\Pi_K)|  = \dim (V_h(K)) - \dim \mathbb P_1(K)\geq 0$ and the equality holds 
only if $K$ is a simplex. The so-called stabilization term is needed to have a 
well-posed discretization. The principle of designing a stabilization is 
two-fold~\cite{Beirao-da-Veiga;Brezzi;Cangiani;Manzini:2012Principles}:  
\begin{enumerate}
\item Consistency. $S_K(u,v)$ should vanish when either $u$ or 
$v$ is in $\mathbb P_1(K)$. This can be ensured to use the slice operator 
$(\operatorname{I}-\Pi_K)$ in the inputs of $S_K(\cdot,\cdot)$ beforehand. 

\item Stability and continuity. $S_K(\cdot,\cdot)$ is chosen so that the following 
norm equivalence holds
\begin{equation}\label{eq:normequivalence}
a(v, v) \lesssim a_h(v, v) \lesssim a(v, v)  \quad \forall v\in V_h.
\end{equation}
\end{enumerate}

The original bilinear form used in 
\cite{Ayuso-de-Dios;Lipnikov;Manzini:2016nonconforming} 
for problem \eqref{eq:dis} is: for $u_h,v_h\in V_h$
$$
a_h^{\text{orig}}(u_h,v_h) := \sum_{K\in \mathcal  T_h} 
\bigl(\nabla \Pi_K u_h, \nabla \Pi_K v_h\bigr)_K + 
\sum_{K\in \mathcal  T_h} 
S_{K}^{\text{orig}}\bigl((\mathrm{I}-\Pi_K) u_h,(\mathrm{I}-\Pi_K) v_h\bigr),
$$
where the stabilization term $S_{K}^{\text{orig}}(\cdot,\cdot)$ penalizes the 
difference between 
the VEM space and the polynomial projection using DoFs \eqref{eq:dofs-face}, while 
gluing the local spaces together using a weak continuity 
condition in \eqref{eq:space-global}: 
for $u_h,v_h\in V_h$
\begin{equation}
\label{eq:sta-orig}
S_K^{\text{orig}} (u_h,v_h) := \sum_{F\in \mathcal{F}_h(K)} h_F^{d-2}\chi_F(u_h) 
\chi_F(v_h)
\end{equation}
The dependence of constants in the norm equivalence \eqref{eq:normequivalence} to 
the geometry of the element $K$ is, however, not carefully studied in literature. 
Especially on anisotropic elements, constants hidden in \eqref{eq:normequivalence} 
could be very large. In 2D and the 3D case when every face $F\in \mathcal{F}_h(K)$ 
is shape-regular, we have the following relation:
\begin{equation}
S_K^{\text{orig}} (u_h,v_h) \eqsim \sum_{F\in \mathcal{F}_h(K)} h_F^{-1} \bigl(Q_F 
u_h, Q_F v_h\bigr)_F. 
\end{equation}
Inspired by this equivalence, we shall use a modified bilinear form: for 
$u_h,v_h\in V_h$
\begin{equation}
\label{eq:bilinear-aniso}
a_h(u_h,v_h) :=  \sum_{K\in \mathcal  T_h} 
\Bigl\{\bigl(\nabla \Pi_K u_h, \nabla \Pi_K v_h\bigr)_K + 
\underbrace{S_{K}\bigl(u_h-\Pi_{\omega_K} u_h,v_h- \Pi_{\omega_K} 
v_h\bigr)}_{(\mathfrak{s})}\Bigr\}.
\end{equation}
In \eqref{eq:bilinear-aniso}, the stabilization on element $K$ is 
\begin{equation}
\label{eq:sta-PiomegaK}
(\mathfrak{s})
:= \sum_{F\in \mathcal{F}_h(K)} h_{\omega_K}^{-1} \bigl(Q_F 
(u_h-\Pi_{\omega_K} u_h), Q_F (v_h- \Pi_{\omega_K} v_h)\bigr)_F,
\end{equation}
which penalizes the difference between a VEM 
function with its projection $\Pi_{\omega_K}$ on the boundary of $K$. To allow faces 
with small $h_{F}$, the weight is changed to $h_{\omega_K}^{-1}$ as well.

Now a nonconforming VEM discretization of~\eqref{eq:model-weak} is: for the bilinear 
form \eqref{eq:bilinear-aniso}, find $u_h\in V_h$ such that
\begin{equation}
\label{eq:VEM}
a_h( u_h, v_h) = \displaystyle\sum_{K\in \mathcal{T}_h} 
\bigl(f, \Pi_K v_h\bigr)_K  \quad \forall v_h \in V_h.
\end{equation} 
In Section \ref{sec:error} we shall derive a general error equation for the 
difference of the VEM approximation $u_h$ to the interpolation $u_I$ under the 
bilinear form induced norm, and present an a priori error bound.

\section{Geometric Assumptions and Inequalities}
\label{sec:geometry}
In this section, we explore some constraints to put on the meshes $\mathcal{T}_h$ in 
order that problem \eqref{eq:VEM} yields a sensible a priori error estimate.

An element $K\in \mathcal{T}_h$ shall be categorized into either 
``isotropic'' or ``anisotropic'' using some of the following assumptions on the 
geometry of the mesh. In the following assumptions, the uniformity of the constants 
is with respect to the mesh size $h \to 0$ in a family of meshes 
$\{\mathcal{T}_h\}$.

\subsection{Isotropic elements}
Firstly, recall that $n_{K}$ represents the number of faces as well as the number of 
DoFs in the element $K$. For both isotropic or anisotropic elements, the following 
assumption shall be fulfilled. 
\smallskip
\begin{itemize}

\item[{\bf A}.] \label{asp:A} For $K\in \mathcal T_h$, the number of faces 
$n_{K}$ is uniformly bounded.

\end{itemize}
\smallskip

Secondly, for a simple polygon/polyhedron that is not self-intersecting, a height 
$l_F$, measuring how far from $F$ one can advance to the interior 
of $K$ in its inward normal direction, determines to what degree of smoothness a 
function defined on $F$ can be extended into the interior of $K$. 

Without loss of generality, the presentation is based on the dimension $d=3$ here, 
after which the case $d=2$ follows naturally. For a given flat face 
$F\in\mathcal{F}_h(K)$, we choose a local Cartesian coordinate $(\xi,\eta,\tau)$ 
such that the face $F$ is on the $\tau =0$ plane. For any $x_F\in F$, $\bm{x}_F 
=  \xi \bm{t}_{F,1} + \eta\bm{t}_{F,2}$, where $\bm{t}_{F,1}$ 
and $\bm{t}_{F,2}$ are two orthogonal unit vectors that span the hyperplane the face 
$F$ lies on. 

The positive $\tau$-direction is chosen such that it is the inward 
normal of $F$. Now define:
\begin{equation}
\label{eq:local-dF}
\delta_F := \inf \Big\{\tau\in \mathbb{R}^+ : \;K\cap \big(F\times 
(\tau,+\infty)\big) = \varnothing \Big\}.
\end{equation} 
As $K$ is a simply polyhedral, $\delta_F > 0$ although it can be very small. 
\begin{figure}[h]
  \centering
  \begin{subfigure}[b]{0.33\linewidth}
    \centering\includegraphics[width=100pt]{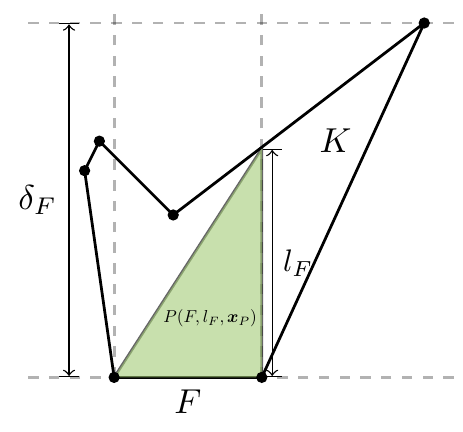}
    \caption{\label{fig:height}}
  \end{subfigure}%
\quad 
\begin{subfigure}[b]{0.33\linewidth}
      \centering\includegraphics[width=100pt]{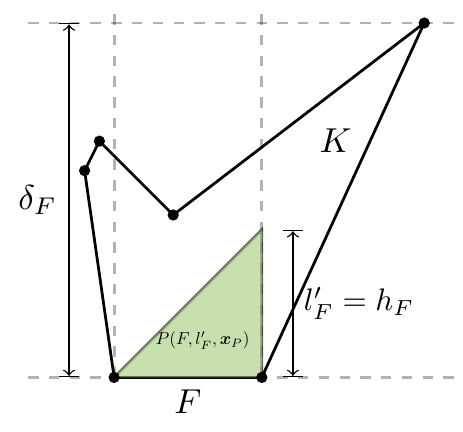}
      \caption{\label{fig:height-rescale}}
\end{subfigure}
\vspace{-10pt}
\caption{(\subref{fig:height}) $l_F\geq \gamma_1 h_F$ with $\gamma_1>1$. 
(\subref{fig:height-rescale}) A rescaled $P(F, l_F', \bs x_P) $ with $l_F' = h_F$.}
\label{fig:localheight}
\end{figure}

A pyramid with base $F$, apex $\bs x_P$, and height $l = 
\operatorname{dist}(\bm{x}_P,F) $ is defined as follows: 
\begin{equation}
\label{eq:pyramid}
P(F,l,\bs x_P) := \{ \bm{x}:\; \bm{x} = (1-t)\bm{x}_F + t \bm{x}_P, t\in (0,1), \bs 
x_F\in F \}.
\end{equation}
Then an inward height $l_F$ associated with face $F$ can be defined as follows:
\begin{equation}
\label{eq:local-lF}
l_F := \sup \Big\{l\in \mathbb{R}^+ : \exists\, P(F, l, \bs x_P) \subset K\cap 
\big(F\times 
(0, \delta_F]\big) \Big\}.
\end{equation}
Here the prism $F\times(0, \delta_F]$ is used to ensure the dihedral angles 
are bounded by $\pi/2$ between $F$ and the side faces of the pyramid 
$P(F,l_F,\bm{x}_P)$.

When $d=2$, as $K$ is non-degenerate (there are no self-intersecting edges) and 
bounded, $0< \delta_F < +\infty$ and $0< l_F \leq  \delta_F$ (see 
Figure.~\ref{fig:height} for example). 
When $d=3$, the 
existence of such pyramid $P(F, l_F, \bs x_P)$ is unclear, since $F$ itself can be 
non-convex. To be able to deal with such case, we impose the following assumption.
\begin{figure}[h]
  \centering
  \begin{subfigure}[b]{0.33\linewidth}
    \centering\includegraphics[width=100pt]{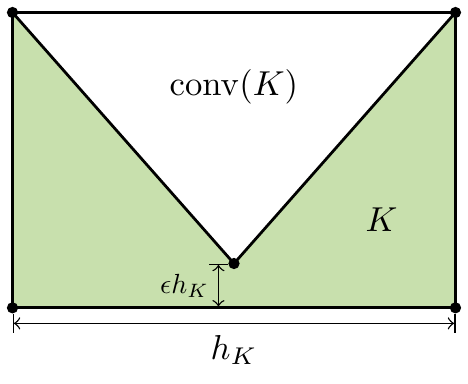}
    \caption{\label{fig:ncelem-hourglass}}
  \end{subfigure}%
\quad
\begin{subfigure}[b]{0.33\linewidth}
      \centering
      \includegraphics[width=75pt,trim=0 0 0.8cm 0]{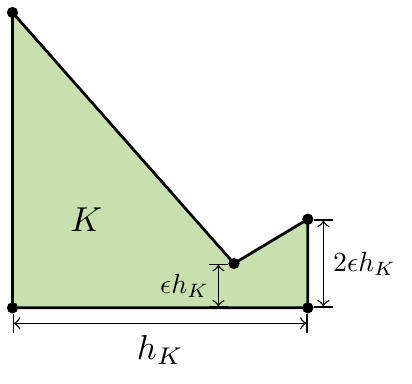}
      \caption{\label{fig:ncelem-epsbump}}
    \end{subfigure}
  \begin{subfigure}[b]{0.3\linewidth}
    \centering
    \centering\includegraphics[width=90pt]{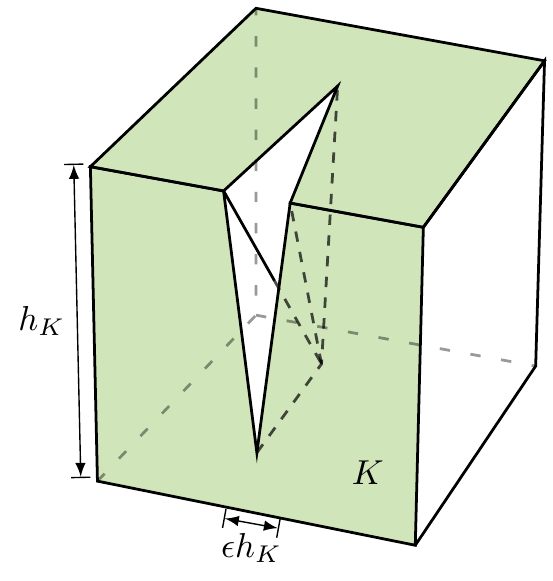}
    \caption{\label{fig:ncelem-crack}}
\end{subfigure}
\vspace{-20pt}
\caption{$\epsilon\to 0$ as $h\to 0$. (\subref{fig:ncelem-hourglass}) $K$ has the 
hourglass shape, and is not 
an isotropic element in the sense of the geometry assumptions in 
\cite{Cao;Chen:2018Anisotropic}. Yet this $K$ is isotropic under Assumptions 
\hyperref[asp:A]{\bf A--B--C}. (\subref{fig:ncelem-epsbump}) 
$K$ has a small hourglass-type bump which is 
ruled out by Assumption \hyperref[asp:C]{\bf C}. (\subref{fig:ncelem-crack}) 
$K$ with a crack is isotropic, and it has two faces satisfying Assumptions 
\hyperref[asp:A]{\bf A--B--C} in the sense of decompositions.}
\label{fig:nconv-example}
\end{figure}
\smallskip
\begin{itemize}
\item[{\bf B}.] (Height condition) \label{asp:B} There exists a 
constants $\gamma_1>0$, such that $\forall F\in\mathcal{F}_h(K)$, it has a partition 
$F = \bigcup_{\beta \in  B_1} F_{\beta}$ with $|B_1|$ uniformly bounded, such that 
each $F_{\beta}$ satisfies the height condition $l_{F_\beta}\geq \gamma_1 
h_F$ and consequently $l_F : = \min_{\beta \in B_1}l_{F_\beta}\geq \gamma_1 
h_F$. 
\end{itemize}
\smallskip

In Figure~\ref{fig:ncelem-hourglass}, the bottom edge satisfies the height condition 
\hyperref[asp:B]{\bf B} only when the decomposition argument is added in the 
assumption. In Figure~\ref{fig:ncelem-crack}, for the whole front face $F$ 
without decomposition, no such pyramid in \eqref{eq:pyramid} exists to yield a 
sensible \eqref{eq:local-lF} since there exists points outside $K$ in the line 
connecting the apex of the pyramid with a point on $F$.

Without loss of generality, one can assume that the constant in Assumption 
\hyperref[asp:B]{\bf B} satisfies $0<\gamma_1\leq 1$ when Assumption 
\hyperref[asp:B]{\bf B} is used as a premise of a proposition in later sections. The 
reason is that, when \hyperref[asp:B]{\bf B} holds, one can always rescale the 
height $l_F$ to $l'_F = \gamma'_1 h_F$, for any $0<\gamma'_1\leq \gamma_1$, while 
the new pyramid $P(F, l'_F, \bs x_P)$ still in $K$. When $\gamma_1> 
1$, we can simply set $\gamma'_1=1$ to be the new $\gamma_1$. See the illustration 
in Figure \ref{fig:height-rescale} for an example in 2-D.

Furthermore, in Figure~\ref{fig:ncelem-hourglass}, it shows a ``good'' 
hourglass-shaped element. To avoid small hourglass-type bumps from an element (e.g., 
see Figure~\ref{fig:ncelem-epsbump}, the following assumption is imposed.
\smallskip
\begin{itemize}
\item[{\bf C}.]  (Hourglass condition) \label{asp:C}$\forall F\in\mathcal{F}_h(K)$, 
it has a partition $F = \bigcup_{\beta \in  B_2} F_{\beta}$ with $|B_2|$ uniformly 
bounded, such that each $F_{\beta}$ satisfies the hourglass condition: $\forall 
\beta\in B_2$, 
there exists 
a convex subset $K_{\beta}\subseteq K$ with $h_{K_{\beta}}\eqsim h_K$, such that 
$P(F_{\beta},l_{F_\beta},\bs x_P)\subset K_{\beta}$. 
\end{itemize}
\smallskip

\begin{figure}[h]
  \centering
  \begin{subfigure}[b]{0.33\linewidth}
    \centering
    \includegraphics[width=100pt]{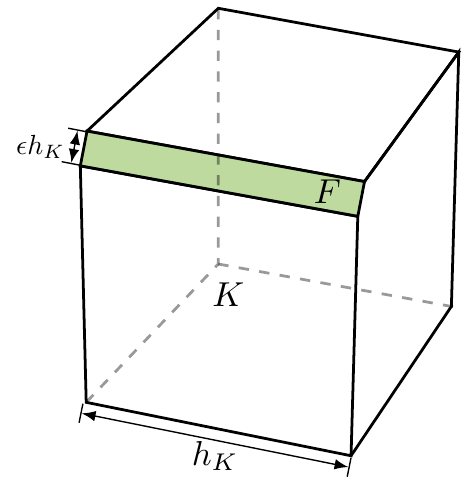}
    \caption{\label{fig:aniso-face}}
  \end{subfigure}%
\begin{subfigure}[b]{0.33\linewidth}
      \centering
      \includegraphics[width=110pt]{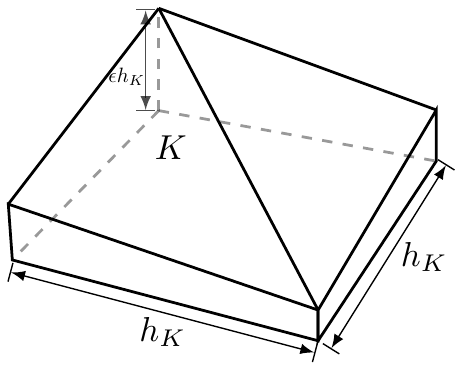}
      \caption{\label{fig:anisoelem-1}}
    \end{subfigure}
  \begin{subfigure}[b]{0.33\linewidth}
    \centering
\includegraphics[width=110pt]{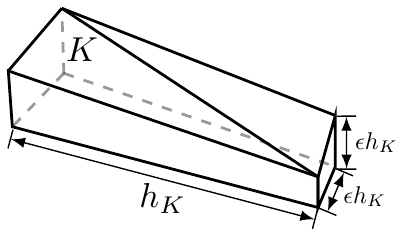}
    \caption{\label{fig:anisoelem-2}}
\end{subfigure}
\vspace{-20pt}
\caption{$\epsilon\to 0$ as $h\to 0$. (\subref{fig:ncelem-hourglass}) $K$ is a cube 
without a prismatic slit, the marked face is 
anisotropic yet the element is isotropic. (\subref{fig:ncelem-epsbump}) 
$K$ itself is anisotropic, all 
four side faces are anisotropic. (\subref{fig:ncelem-crack}) 
$K$ is anisotropic.}
\label{fig:anisoelem}
\end{figure}

Now we say an element $K$ is {\em isotropic} if Assumptions \hyperref[asp:A]{\bf 
A--B--C} hold for $K$, \SC{with the partitions $\{ F_{\beta}\}_{\beta\in B_1} 
= \{ F_{\beta}\}_{\beta\in B_2}$ for the same face $F$ in Assumptions 
\hyperref[asp:B]{\bf B--C} }.
Otherwise it is called {\em anisotropic}. As we mentioned earlier, isotropy and 
anisotropy can be formulated for 2-D polygons using the height condition and 
hourglass condition for edges. A complication in 3-D meshes is that for an 
isotropic polyhedron, we may have an anisotropic face or a tiny face. In both cases, 
$|F| \ll h_K^2$ (see Figure~\ref{fig:anisoelem} for examples of polyhedral 
elements). Henceforth, when Assumptions \hyperref[asp:B]{\bf B} and/or 
\hyperref[asp:C]{\bf C} are met, we denote $P_F:=P(F,l_F,\bs x_P)$, and whether the 
decomposition is used or not should be clear from the context.

\begin{lemma}[\revision{Scale of the volume/area for isotropic elements}]
\label{lemma:iso}
If $K$ is isotropic in the sense of Assumptions \hyperref[asp:A]{\bf 
A--B--C}, then $|K|\eqsim h_K^d$.
\end{lemma}
\begin{proof}
Obviously $|K| \lesssim h_K^d$ by the definition of diameter. It suffices to bound 
the volume $|K|$ below by $h_K^d$. We choose the face $F$ with the largest area on 
$\partial K$, by Assumption 
\hyperref[asp:A]{\bf A}, $|F| \gtrsim |\partial K|$. With slightly abuse of the 
order of the presentation, by the trace inequality with $v= 1$ in Lemma 
\ref{lemma:trace-hK}, we have $|F| = \norm{v}_{0,F}^2 \lesssim h_K^{-1} 
\norm{v}_{0,K}^2 = h_K^{-1} |K|$. Hence $|K|\gtrsim h_K |F| \gtrsim h_K |\partial 
K|$, and the lemma follows from the isoperimetric inequality $|\partial K| 
\gtrsim |K|^{(d-1)/d}$.  
\end{proof}

\subsection{Anisotropic elements}
For anisotropic elements, by definition, there exists faces such that the 
height condition and/or hourglass condition are violated.  
The case $l_F\ll h_F$ can be caused by either the non-convexity of 
$F$ and $K$, or the chunkiness parameter of $K$ being large.

To be able to use the trace inequalities on a face in an anisotropic 
element, the following condition on this element 
$K$ is proposed:

\smallskip
\begin{itemize}
\item[{\bf D}.] \label{asp:D} There exists an isotropic extended element patch 
$\omega_K$ consists of elements in $\mathcal T_h$ such that 
\begin{enumerate}
\item $K\subseteq \omega_K$;
\item $h_{\omega_K} \leq \gamma_2 h$ with a uniform constant $\gamma_2>0$;
\item $\forall F\in \mathcal{F}_h(K)$, $F$ satisfies Assumptions 
\hyperref[asp:B]{{\bf B}} and \hyperref[asp:C]{{\bf C}} toward 
$\omega_K$;
\item $n_{\omega_K}:= \bigl|\{K'\in \mathcal T_h( \omega_K) \} \bigr|$ is uniformly 
bounded above.
\end{enumerate}
\end{itemize}
\smallskip

By the construction of $\omega_K$ and the definition of an isotropic element, the 
height condition in Assumption \hyperref[asp:B]{{\bf B}} and the hourglass 
condition 
in Assumption \hyperref[asp:C]{{\bf C}} are met for every face $F\in 
\mathcal{F}_h(\partial \omega_K)$. With Assumption \hyperref[asp:D]{{\bf D}}, one 
can then lift a function defined on a boundary face $F\in \mathcal{F}_{h}({K})$  to 
the isotropic element $\omega_K$. 

\subsection{Finite overlapping of convex hulls}
For polytopal meshes, we impose the following conditions on the convex hull of $K$ 
for isotropic elements or $\omega_K$ for anisotropic elements. 
\smallskip
\begin{itemize}
\item[{\bf E}.] \label{asp:E} There exists a uniform constant $\gamma_3>0$ such that 
for each $K\in \mathcal T_h$
$$
| \{ K'\in \mathcal T_h: \operatorname{conv}(\omega_{K'})\cap  
\operatorname{conv}(\omega_{K}) \neq \varnothing \}| \leq \gamma_3.
$$
\end{itemize}
It can be verified that Assumption \hyperref[asp:E]{{\bf E}} is ensured if for any 
vertex, there are uniformly bounded number of polytopal elements 
surrounding this vertex. 

\subsection{Trace inequalities}
When using a trace inequality, one should be extremely careful as the constant  
depends on the shape of the domain. In this subsection, we shall re-examine 
several trace inequalities with more explicit analyses on the geometric conditions. 

\begin{lemma}[A trace inequality on a face in a polytopal element]
\label{lemma:trace-l2}
Suppose for the face $F$, there exists a triangle/pyramid 
$P_F:=P(F,l_F,\bs x_P)\subset K\cap \big(F\times (0, \delta_F]\big)$ with 
height $l_F$, then the following trace inequality holds:
\begin{equation}
\label{eq:tr-l2}
\norm{v}_{0,F} \lesssim l_F^{-1/2} \norm{v}_{0,P_F} 
+ \big(h_Fl_F^{-1/2} + l_F^{1/2} \big) \norm{\nabla v}_{0,P_F}.
\end{equation}
Consequently if furthermore the height condition \hyperref[asp:B]{{\bf B}} is 
satisfied, it holds that for $P_F := \bigcup_{\beta \in B_1} P_{F_{\beta}}$
\begin{equation}\label{eq:tr-h}
\norm{v}_{0,F} \lesssim h_F^{-1/2} \norm{v}_{0,P_F} + h_F^{1/2} 
\norm{\nabla v}_{0,P_F}.
\end{equation}
\end{lemma}
\begin{proof}
We first consider the case $l_F \eqsim h_F$ when $d=3$. 
By Lemma A.3 in~\cite{Wang;Ye:2014Galerkin},
\begin{equation}
\label{eq:tr-pyra}
\norm{v}_{0,F}^2 \lesssim h_F^{-1}\norm{v}_{0,P_{\frac1 2}(F, l_F, \bs x_P)}^2 + h_F 
\norm{\nabla v}_{0,P_{\frac1 2}(F, l_F, \bs x_P)}^2,
\end{equation}
where $P_{\frac12}(F, l_F, \bs x_P):=  \{ \bm{x}:\; \bm{x} = (1-t)\bm{x}_F + t 
\bm{x}_P, t\in (0,1/2), \bs x_F\in F \}$. The motivation to truncate the pyramid 
$P_F$ to the prismatoid is that the Jacobian of the mapping from the 
prismatoid to the prism is bounded.

For a general case, without loss of generality, we assume 
$l_F\leq h_F$, since otherwise, one can set $l_F = h_F$ first and \eqref{eq:tr-l2} 
still holds by \eqref{eq:tr-pyra}: we consider the following mapping
$(\bs x_F, \tau) \mapsto (\bs x_F, \tau l_F /h_F)$, denote $P := P_{\frac12}(F, l_F, 
\bs x_P)$, and let $\nabla_{\bm{x}_F}$ be the gradient taken with respect to 
$(\xi,\eta)$ in $F$'s local coordinate system. Then a straightforward change of 
variable computation yields:
\begin{equation}\label{eq:Ftrace}
\norm{v}_{0,F}^2 \lesssim l_F^{-1} \norm{v}_{0, P}^2 
+ h_F^2 l_F^{-1} \norm{\nabla_{\bm{x}_F} v}_{0,P}^2
+ l_F \norm{\partial_{\tau} v}_{0,P}^2.
\end{equation}
When $d=2$, a similar scaling argument can be found 
in \cite[Lemma 6.3]{Cao;Chen:2018Anisotropic} and estimate \eqref{eq:Ftrace} changes 
to $\norm{v}_{0,e}^2 \lesssim l_e^{-1} \norm{v}_{0,P}^2 
+ h_e^2 l_e^{-1} \norm{\partial_x v}_{0,P}^2 + l_e \norm{\partial_y v}_{0,P}^2$ for 
an edge $e$. As a result, \eqref{eq:tr-l2} holds. 

When $F$ satisfies Assumption \hyperref[asp:B]{{\bf B}}, 
$F = \bigcup_{\beta \in B_1} F_{\beta}$, each of $F_{\beta}$ satisfies the height 
condition 
with disjoint pyramids $P(F_{\beta},l_{F_\beta}, \bs x_{\beta})$. One can 
rescale all $l_{F_{\beta}}$ to be 
$l_F:= \min\limits_{\beta \in B_1} l_{F_{\beta}}$, and 
$P(F_{\beta},l_F)\subset P(F_{\beta},l_{F_{\beta}})\subset K\cap 
\big(F_{\beta}\times (0, \delta_{F_{\beta}}]\big)$. Thus under Assumption 
\hyperref[asp:B]{{\bf B}}, $\norm{v}_{0,F}^2=\sum_{\beta\in 
B_1}\norm{v}_{0,F_{\beta}}^2$ can be estimated by a simple summation of 
\eqref{eq:Ftrace}. 
\end{proof}

As we mentioned before, even for an isotropic element, it may contain a face $F$ 
with $h_F\ll h_K$ and thus the factor $h_F^{-1/2}$ in the trace inequality 
\eqref{eq:tr-h} may be uncontrollable. Next we shall use the hourglass condition 
\hyperref[asp:C]{{\bf C}} to replace $h_F^{-1/2}$ by a smaller factor $h_K^{-1/2}$.

\begin{lemma}[\revision{A trace inequality on a face satisfying the height condition 
and the 
hourglass condition}]
\label{lemma:trace-hK}
If a face $F\in \mathcal{F}_h(K)$ satisfies the height condition 
\hyperref[asp:B]{{\bf B}} and the hourglass condition \hyperref[asp:C]{{\bf C}}, 
then it holds that
\begin{equation}\label{eq:tr-hK}
\norm{v}_{0,F} \lesssim h_K^{-1/2} \norm{v}_{0,K} + h_K^{1/2} 
\norm{\nabla v}_{0,K}.
\end{equation}
\end{lemma}
\begin{proof}
{As the final inequality \eqref{eq:tr-hK} can be trivially generalized from each 
$F_{\beta}$ to $F = \bigcup_{\beta}F_{\beta}$ using the same argument with the one 
in Lemma \ref{lemma:trace-l2}, we consider only one face $F_{\beta}$ in the 
decomposition, which shall be denoted by $F$ subsequently in the proof.}  First 
Assumption 
\hyperref[asp:B]{{\bf B}} implies the validity of the trace inequality 
\eqref{eq:tr-h}. 
When Assumption \hyperref[asp:C]{{\bf C}} is met, let $K_F$ be 
the convex subset of $K$ containing $P_F$. Since \eqref{eq:tr-hK} holds trivially if 
$h_F\eqsim h_K$, it suffices to consider the case when $h_F\ll h_K$. Due to the 
convexity of $K_F$, $\operatorname{conv}(P_F) \subset K_F$, without loss of 
generality we can assume that $P_F$ is convex. 
Moreover, we recall that the rescaling argument facilitated by Assumption 
\hyperref[asp:B]{{\bf B}} allows us to set $h_{P_F}\eqsim h_F$. By $K_F\subseteq K$, 
it suffices to show that:
\begin{equation}
\label{eq:tr-embed}
h_F^{-1}\norm{v}_{0,P_F}^2 \lesssim h_{K}^{-1} \norm{v}_{0,K_F}^2 
+ h_{K} \norm{\nabla  v}_{0,K_F}^2.
\end{equation}

By Assumption \hyperref[asp:C]{{\bf C}}, there 
exists a point $\bm{a}\in K_F$, such that $\operatorname{dist}(\bm{a}, P_F) \eqsim 
h_{K_F}\eqsim h_K$ and $\operatorname{conv}(\bm{a},P_F)\subset K_F$ (e.g., see 
Figure \ref{fig:trace-ineq}). Now 
thanks to the convexity of $P_F$, it has the following local polar coordinate 
representation using $\bm{a}$ as the origin:
\begin{equation}
P_F = \{ \bm{x} = \bm{x}(r,\bm{\omega}) = r\bm{\omega}: r_1(\bm{\omega}) \leq r \leq 
r_2(\bm{\omega}), \bm{\omega} \in A_F\subset\mathbb{S}^{d-1}, d=2,3 \},
\end{equation}
and $\operatorname{conv}(\bm{a},P_F) = \{ r\bm{\omega}: 
0 \leq r \leq r_2(\bm{\omega}), \;\bm{\omega} \in A_F \}$. 
\begin{figure}[h]
\centering
\includegraphics[width = 0.45\textwidth]{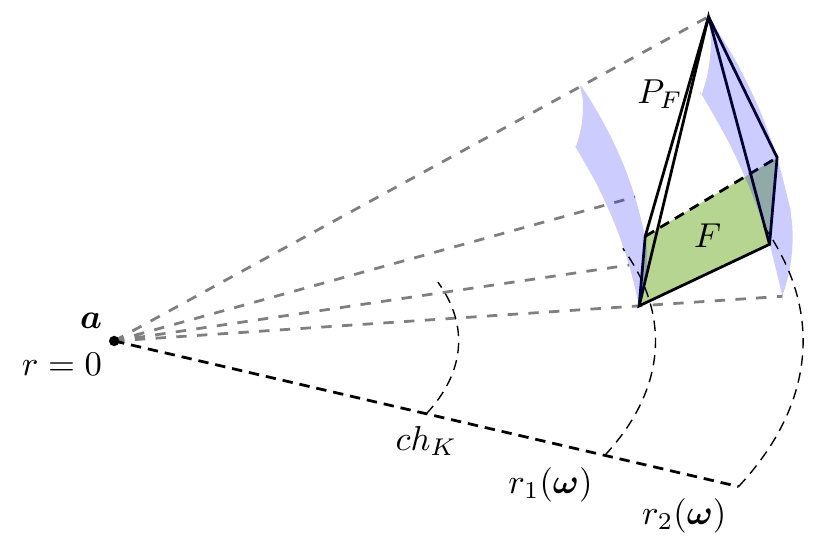}
\caption{An illustration of $\operatorname{conv}(\bm{a},P_F) \subset K_F$.}
\vspace{-20pt}
\label{fig:trace-ineq}
\end{figure}
For $\bm{x}(r,\bm{\omega})\in 
P_F$, we denote $v(r,\bm{\omega}):= 
v\big(\bm{x}(r,\bm{\omega})\big)$. Now for a fixed surface variable 
$\bm{\omega}$, $|r_1(\bm{\omega}) - r_2(\bm{\omega})| \lesssim h_F$, and
$r_i(\bm{\omega}) \eqsim h_K$. Moreover, we can choose $\rho$ and constant 
$c$ bounded away from $0$ independent of $h_K$ such that $ch_K = \rho < 
r_1(\bm{\omega})$, and thus $|r_2(\bm{\omega}) - \rho|\gtrsim h_K$. 
The mean value theorem implies that there exists a $\xi\in (\rho, r_2(\bm{\omega}))$ 
such that
\begin{equation}
v^2(\xi,\bm{\omega}) = \frac{1}{|r_2(\bm{\omega}) - \rho|}
\int^{r_2(\bm{\omega})}_{\rho} v^2(t,\bm{\omega}) \dd t
\lesssim \frac{1}{h_K} \int^{r_2(\bm{\omega})}_{\rho} v^2(t,\bm{\omega}) \dd t.
\end{equation}
By the fundamental theorem of calculus, 
Young's inequality, and the inequality above, we have
\begin{equation*}
\begin{aligned}
v^2(r,\bm{\omega}) &= v^2(\xi,\bm{\omega}) + \int^r_{\xi} \partial_t\big( 
v^2(t,\bm{\omega}) \big) \dd t
\leq  v^2(\xi,\bm{\omega}) 
+ h_K^{-1}\int^r_{\xi}  v^2  \dd t 
+ h_K \int^r_{\xi} |\partial_t v|^2  \dd t
\\
& \lesssim  h_K^{-1}\int^{r_2(\bm{\omega})}_{\rho}  v^2  \dd t 
+ h_K \int^{r_2(\bm{\omega})}_{\rho} |\partial_t v|^2  \dd t,
\end{aligned}
\end{equation*}
where we note that $\xi$ is not present in the final inequality above, thus this 
inequality holds for any $\bm{\omega} \in A_F$.
Integrating above inequality with respect to $r^{d-1}\dd r$ and using the fact that 
$|r_1(\bm{\omega})- r_2(\bm{\omega})| \leq h_{P_F}\eqsim h_F$, we have:
\begin{align*}
\int^{r_2(\bm{\omega})}_{r_1(\bm{\omega})} v^2(r,\bm{\omega}) r^{d-1}\dd r 
& \lesssim \int^{r_2(\bm{\omega})}_{r_1(\bm{\omega})} 
\left(h_K^{-1}\int^{r_2(\bm{\omega})}_{\rho}  
v^2  \dd t + h_K \int^{r_2(\bm{\omega})}_{\rho} |\partial_t v|^2  \dd t 
\right)r^{d-1}dr
\\
& \lesssim h_F h_K^{d-1}\left(h_K^{-1}\int^{r_2(\bm{\omega})}_{\rho}  
v^2  \dd t + h_K \int^{r_2(\bm{\omega})}_{\rho} |\partial_t v|^2  \dd t 
\right):= (\dagger)
\end{align*}
As $t > \rho = ch_K$ in the integrals and $c$ is bounded away from $0$, the factor 
$h_K^{d-1}$ can be moved into the integrals above, thus $(\dagger)$ can be 
bounded by
\begin{equation*}
(\dagger) \lesssim h_F 
\left(h_K^{-1}\int^{r_2(\bm{\omega})}_{\rho}  v^2(t,\bm{\omega}) t^{d-1} \dd t  
+ h_K\int^{r_2(\bm{\omega})}_{\rho} |\partial_t v(t,\bm{\omega})|^2 t^{d-1} \dd t 
\right).
\end{equation*}
Lastly, $\bm{x} = r\bm{\omega}$ together with $|\bm{\omega}|=1$ implies $|\partial_t 
v(t,\bm{\omega})| \leq |\nabla v|$, integrating both sides of the integral above 
with respect to surface measure $\dd \bm{\omega}$ on $A_F$ and rearranging the 
factors yield:  
\begin{equation*}
h_F^{-1}\norm{v}_{0,P_F}^2 \lesssim h_K 
\norm{v}_{0,\operatorname{conv}(\bm{a},P_F)}^2 
+ h_K\norm{\nabla v}_{0,\operatorname{conv}(\bm{a},P_F)}^2.
\end{equation*}
Consequently \eqref{eq:tr-embed} is valid since 
$\operatorname{conv}(\bm{a},P_F)\subset K_F$, and the lemma follows.
\end{proof}

\begin{remark}\rm 
\label{remark:trace}
\revision{When $K$ is uniformly star-shaped, we can choose the vertex $\bs a$ as the 
center of the largest inscribed ball for all faces $F$. With Assumptions 
\hyperref[asp:B]{{\bf B}} and  \hyperref[asp:C]{{\bf C}}, the vertex $\bs a$ could 
vary for different faces, and thus a more flexible geometry is allowed. See 
Fig.~\ref{fig:ncelem-hourglass} and \ref{fig:ncelem-crack} for examples satisfying 
Assumptions \hyperref[asp:A]{{\bf A}}, \hyperref[asp:B]{{\bf B}} and  
\hyperref[asp:C]{{\bf C}} but not uniformly star-shaped.
}
\end{remark}

\subsection{Poincar\'e Inequalities}
In this subsection, we review Poincar\'{e}--Friedrichs inequalities with a constant 
depending only on the diameter of the domain but not on the shape.

\begin{lemma}[Poincar\'e inequality of a linear polynomial on a face]
\label{lemma:poincare-linear-face}
On any face 
$F\in \mathcal{F}_h(K)$, for a linear polynomial $q\in \mathbb P_1(F)$, the estimate
$$
\norm{q- \overline{q}^{F}}_{0,F} \leq h_F\norm{\nabla_F q}_{0,F},
$$
where $\nabla_F$ denotes the surface gradient on $F$.
\end{lemma}
\begin{proof}
Here we use the local Cartesian coordinate $\xi \bm{t}_{F,1} + \eta\bm{t}_{F,2}=: 
\bm{x}\in F$ in defining \eqref{eq:local-dF}, then for $q\in \mathbb P_1(F)$, 
$q = \bm{x}\cdot \nabla_F q + c$, where $\nabla_F q$ is a constant vector. The lemma 
then follows from the a direct calculation: $q- \overline{q}^{F} = (\bm{x} - 
\overline{\bm{x}}^F )\cdot \nabla_F q$, and
\begin{equation}
\norm{q- \overline{q}^{F}}_{0,F}^2 \leq \int_F |\bm{x} - \overline{\bm{x}}^F|^2 \,
|\nabla_F q|^2\,\dd S \leq h_F^2 \norm{\nabla_F q}_{0,F}^2.
\end{equation}
\end{proof}

\begin{lemma}[Poincar\'e inequality of a linear polynomial on the patch] 
\label{lemma:poincare-linear-patch}
For a linear polynomial $q\in \mathbb P_1(\omega_K)$ such that 
$\int_{\partial \omega_K} q \,\dd S = 0$, where $\omega_K$ satisfies 
Assumption D, the following estimate holds with a constant independent of the 
geometries of $K$ or $\omega_K$:
$$
\norm{q}_{0,K} \leq h_{\omega_K}\|\nabla q\|_{0,K}.
$$
\end{lemma}
\begin{proof}
For $q\in \mathbb P_1(\omega_K)$, $q =\bm{x} \cdot \nabla q + c$ with a constant 
$c$. By the fact that the 
constraint is imposed on the boundary integral on $\partial\omega_K$, similar to the 
previous lemma, it can be verified that $q = (\bm{x} - \overline{\bm{x}}^{\partial 
\omega_K})\cdot \nabla q$, where 
$$\overline{\bm{x}}^{\partial \omega_K} = 
\frac{1}{|\partial \omega_K|} \int_{\partial \omega_K} \bm{x} \dd S = \sum_{F\in 
\mathcal{F}_h(\partial \omega_K)} \frac{|F|}{|\partial \omega_K|} 
\overline{\bm{x}}^F,$$ 
where $\overline{\bm{x}}^F\in \operatorname{conv}(F)$, hence
$\overline{\bm{x}}^{\partial \omega_K}\in \operatorname{conv}(\omega_K)$. As a 
result, $|\bm{x} - \overline{\bm{x}}^{\partial 
\omega_K}| \leq h_{\omega_K}$, and we have
\begin{equation}
\norm{q}_{0,K}^2 \leq \int_K |\bm{x} - \overline{\bm{x}}^{\partial 
\omega_K}|^2 \, |\nabla q|^2\,\dd \bm{x} \leq h_{\omega_K}^2 \norm{\nabla q}_{0,K}^2.
\end{equation}
\end{proof}

Notice that the constraint is imposed on the boundary of a bigger patch $\omega_K$ 
but the inequality holds on a smaller region $K$. When using this inequality in the 
a priori error estimate in order to get the optimal rate of convergence, the 
constant will dependent only on $\gamma_2$.

For the approximation property of the polynomial projection, we opt to use the 
Poincar\'e inequality on a convex domain, thus to utilize the convex hull of a 
possible non-convex element.
\begin{lemma}[Poincar\'{e} inequality on the convex hull]
\label{lemma:poincare-convex}
Let $\omega$ be a \revision{bounded simple polygon/polyhedron}, the following 
Poincar\'{e} inequality holds for any $v\in 
H^1\big(\operatorname{conv}(\omega)\big)$:
\begin{equation}
\label{eq:poincare}
\norm{v- \overline{v}^{\omega}}_{0,\omega}\leq \frac{h_\omega}{\pi}\norm{\nabla 
v}_{0,\operatorname{conv}(\omega)}.
\end{equation}
\end{lemma}
\begin{proof}
 As $\overline{v}^{\omega}$ is the best constant approximation in $L^2(\omega)$-norm:
\begin{equation*}
\norm{v - \overline{v}^{\omega}}_{0, \omega}
\leq \big\| v - \overline{v}^{\operatorname{conv}(\omega)} \big\|_{0, K}
\leq \big\| v - \overline{v}^{\operatorname{conv}(\omega)} \big\|_{0, 
\operatorname{conv}(\omega)}
\leq  \frac{h_\omega}{\pi}\|\nabla v\|_{0, \operatorname{conv}(\omega)}.
\end{equation*}
In the last step, the Poincar\'e inequality on a convex 
set~\cite{Payne;Weinberger:1960optimal} is used.
\end{proof}

We then establish a similar result when the constraint is posed on the boundary 
integral.

\begin{lemma}[Poincar\'{e} inequality with zero boundary average on isotropic 
polyhedron]
\label{lemma:poincare-b}
Let $K$ be a polypotal element satisfying Assumptions \hyperref[asp:A]{\bf A--B}, 
then for any $v\in H^1(\operatorname{conv}(K))$, the following Poincar\'{e} 
inequality holds:
\begin{equation}
\label{eq:poincare-b}
\big\Vert{v- \overline{v}^{\partial K}}\big\Vert_{0,K}
\lesssim h_K\norm{\nabla v}_{0,\operatorname{conv}(K)}.
\end{equation}
\end{lemma}
\begin{proof}
First triangle inequality implies
\begin{equation}
\label{eq:poincare-b1}
\big\Vert{v- \overline{v}^{\partial K}}\big\Vert_{0,K}\leq 
\norm{v-\overline{v}^{K}}_{0,K}   
+ \big\Vert{\overline{v}^{K}- \overline{v}^{\partial K}}\big\Vert_{0,K},
\end{equation}
where the first term can be estimated by Lemma \ref{lemma:poincare-convex}. 
Rewriting the second term above and using the Cauchy--Schwarz inequality yield:
\begin{equation}
\label{eq:poincare-b2}
\begin{aligned}
 \big\Vert{\overline{v}^{K}- \overline{v}^{\partial K}}\big\Vert_{0,K}
&= |K|^{1/2}\left|\frac{1}{|\partial K|}
\int_{\partial K} (\overline{v}^{K} - v ) \dd S  \right|
\\
&\leq 
 \frac{|K|^{1/2}}{|\partial K|}  \sum_{F\in \mathcal{F}_h(K)}
 |F|^{1/2}\norm{v- \overline{v}^{K}}_{0,F}.
\end{aligned}
\end{equation}
By the trace inequality in Lemma \ref{lemma:trace-l2}
\begin{equation}
\norm{v- \overline{v}^{K}}_{0,F}
\leq h_F^{-1/2} \norm{v- \overline{v}^{K}}_{0,K} 
+ h_F^{1/2} \norm{\nabla v}_{0,K}.
\end{equation}
Applying the Poincar\'{e} inequality in 
Lemma \ref{lemma:poincare-convex} on $\norm{v- \overline{v}^{K}}_{0,K}$ and the fact 
that $h_F\leq h_K$ yields:
\begin{equation}
\label{eq:poincare-1}
 \big\Vert{\overline{v}^{K}- \overline{v}^{\partial K}}\big\Vert_{0,K}
\leq
\frac{h_K|K|^{1/2}}{|\partial 
K|} \sum_{F\in \mathcal{F}_h(K)} \left(\frac{|F|}{h_F}\right)^{1/2}
\norm{\nabla v}_{0,\operatorname{conv}(K)}.
\end{equation}
As $|F|\lesssim h_F^{d-1}$ and $|\partial K| = \sum_{F\in \mathcal{F}_h(K)} |F|$,
$$
\sum_{F\in \mathcal{F}_h(K)} \left(\frac{|F|}{h_F}\right)^{1/2} \lesssim \sum_{F\in 
\mathcal{F}_h(K)} |F|^{\frac{d-2}{2(d-1)}} \leq n_K^{\frac{d}{2(d-1)} } 
|\partial K|^{\frac{d-2}{2(d-1)}}.
$$
Then
$$
\frac{|K|^{1/2}}{|\partial K|} \sum_{F\in \mathcal{F}_h(K)} 
\left(\frac{|F|}{h_F}\right)^{1/2} \lesssim C(n_K) 
\frac{|K|^{1/2}}{|\partial K|^{d/2(d-1)}} \lesssim C(n_K),
$$
where in the last step, we have used the isoperimetric inequality $|K| \leq C_d 
|\partial K|^{d/(d-1)}$. 
\end{proof}

\section{A Priori Error Analysis}
\label{sec:error}
The error analysis will be performed under a mesh dependent norm $\tnorm{\cdot}$ 
induced by $a_h(\cdot,\cdot)$, i.e., for $v\in H_0^1(\Omega) + V_h$
\begin{equation}
\label{eq:norm-ah}
\tnorm{v}^2 := a_h(v, v) = \sum_{K\in \mathcal  T_h} \left ( \norm{\nabla \Pi_K 
v}^2_K 
+ h_{\omega_K}^{-1}\sum_{F\in \mathcal F_h(K)}\norm{Q_F(v-\Pi_{\omega_K} v)}^2_F 
\right ).
\end{equation}
which is weaker than the $H^1$-seminorm $|\cdot|_{1}$ upon which the conventional 
VEM analysis is built. We denote the local norm on $K$ as $\tnorm{\cdot}_K$. As all 
projections $\Pi_K, \Pi_{\omega_K},$ and $Q_F$ can 
be computed using only on the DoFs (see \eqref{eq:rhspr} and 
\eqref{eq:pr-pw}), it is 
straightforward to verify that $\tnorm{v - v_I} = 0$ for the interpolant $v_I$ 
defined using DoFs in \eqref{eq:intp}.

\subsection{A mesh dependent norm}
Firstly, the following 
lemma is needed for proving $\tnorm{\cdot}$ is a norm on $V_h$ which bounds the 
projection $\Pi_{\omega_K}$ in an extended element 
measured in the $H^1$-seminorm.

\begin{lemma}[Bound of the projection $\Pi_{\omega_K}$ on the patch]
\label{lemma:pr-split}
Let $\overline{\omega} = \overline{\cup_{\alpha\in A}K_{\alpha}}$, for $v\in 
H^1(\omega)$, the following estimate holds:
\begin{equation}
\label{eq:pr-split}
\norm{\nabla \Pi_{\omega} v}_{0,\omega}^2 \leq \sum_{\alpha\in A}
\norm{\nabla \Pi_{K_{\alpha}} v}_{0,K_{\alpha}}^2. 
\end{equation}
\end{lemma}
\begin{proof}
By definition \eqref{eq:pr}, for any $q\in \mathbb{P}_1(\omega)$, 
$q\big\vert_{K_{\alpha}} \in  \mathbb{P}_1(K_{\alpha})$, thus by \revision{the 
Cauchy-Schwarz inequality and an $\ell^2$--$\ell^2$ H\"{o}lder inequality}, we have
\begin{equation}
\begin{aligned}
& \left(\nabla \Pi_{\omega}  v, \nabla q\right)_{\omega} = 
\left(\nabla  v , \nabla q\right)_{\omega} 
= \sum_{\alpha\in A} \left(\nabla  v , \nabla q\right)_{K_{\alpha}}
= \sum_{\alpha\in A} 
\left(\nabla \Pi_{K_{\alpha}}  v , \nabla q\right)_{K_{\alpha}}
\\
\leq  & \sum_{\alpha\in A} 
\norm{\nabla \Pi_{K_{\alpha}}  v}_{0,K_{\alpha}}  \norm{\nabla q}_{0,K_{\alpha}}
\leq \left(\sum_{\alpha\in A} 
\norm{\nabla \Pi_{K_{\alpha}}  v}_{0,K_{\alpha}}^2 \right)^{1/2}\norm{\nabla 
q}_{0,\omega}.
\end{aligned}
\end{equation}
The lemma then follows from letting $q=\Pi_{\omega} v$. 
\end{proof}

\begin{lemma}
\label{lemma:norm}
$\tnorm{\cdot}$ defines a norm on the nonconforming VEM space $V_h$.
\end{lemma}
\begin{proof}
Since each component of $\tnorm{\cdot}$ supports the triangle inequality and is 
scalable, it suffices to verify that if $\tnorm{v_h} = 0$ for $v_h\in V_h$,  
then $v_h\equiv 0$. By definition, $a_h(v_h,v_h)=\tnorm{v_h}^2 = 0$ implies that
\begin{equation}
\label{eq:norm-1}
\nabla \Pi_K  v_h =\bm{0}, \; \forall K\in \mathcal{T}_h; \quad 
Q_F(v_h- \Pi_{\omega_K} v_h) = 0 \text{ on } F, \; \forall F\in 
\mathcal{F}_h(K).
\end{equation}
Without the loss of generality, we assume that $\omega_K$ consists $K$ and $K'$ 
sharing a face, which covers the case of $\omega_K = K$ while can be generalized to 
the case where $\omega_K$ contains three or more elements. 

Firstly by Lemma \ref{lemma:pr-split}, $\nabla \Pi_{\omega_K} v_h = \bm{0}$ 
since $\nabla \Pi_{K} v_h = \nabla \Pi_{K'} v_h=\bm{0}$. Restricting 
ourselves on $K$, consider the following quantity:
\begin{equation}
\begin{aligned}
& \;\norm{\nabla v_h}_{0,K}^2  = 
\bigl(\nabla v_h, \nabla(v_h -\Pi_{\omega_K} v_h) \bigr)_K
\\
& = -\bigl(\Delta v_h, v_h - \Pi_{\omega_K} v_h \bigr)_K+ 
\left\langle \nabla v_h\cdot \bm{n}, v_h - \Pi_{\omega_K} v_h 
\right\rangle_{\partial K}
\\
&= \sum_{F\in \mathcal{F}_h(K)} 
\bigl(\nabla v_h\cdot \bm{n}, 
Q_F(v_h - \Pi_{\omega_K} v_h) \bigr)_{F} = 0.
\end{aligned}
\end{equation}
In the last step, $\Delta v_h = 0$ in $K$ is used. Since $\nabla v_h\cdot \bm{n}\in 
\mathbb{P}_{0}(F)$, the $L^2$ projection $Q_F$ can be inserted into the pair. 

As a result of \eqref{eq:norm-1}, in every $K$, $\nabla v_h = \bm{0}$ thus $v_h = 
\text{constant}$. Finally, by the boundary condition and the continuity condition in 
\eqref{eq:space-global}, $v_h \equiv 0$. 
\end{proof}

\subsection{\revision{A priori error estimates on isotropic elements}}
Next, an error equation is developed for the lowest order nonconforming VEM 
following~\cite{Cao;Chen:2018Anisotropic,Mu;Wang;Ye:2015Galerkin}, and the a priori 
error analysis on isotropic elements is established.

\begin{lemma}[An error equation]
\label{lemma:error-omega}
Let $u_h$ and $u_I$ be the solution to problem \eqref{eq:VEM} and the 
canonical interpolation in \eqref{eq:intp} respectively, and let $u_{\pi}$ be any 
piecewise linear polynomial on $\mathcal T_h$, for any $v_h\in V_h$ and  
stabilization $S_K(\cdot,\cdot)$, it holds that
\begin{equation}
\label{eq:error}
\begin{aligned}
a_h(u_h-u_I,v_h)  =& \sum_{K\in \mathcal{T}_h}\sum_{F\in \mathcal{F}_h(K)} 
\left\langle \nabla (u -u_{\pi} )\cdot \bm{n}, 
Q_F v_h - \Pi_K v_h\right\rangle_{F}
\\
\quad & -  \sum_{K\in \mathcal{T}_h} 
S_{K}\bigl(u_I-\Pi_{\omega_K} u_I, v_h - \Pi_{\omega_K} v_h\bigr).
\end{aligned}
\end{equation}
\end{lemma}
\begin{proof}
Using the VEM discretization problem \eqref{eq:VEM}, the original PDE $-\Delta u = 
f$, the definition of the elliptic projection \eqref{eq:pr}, and 
the integration by parts, we have
\begin{equation}
\label{eq:err-intbyparts}
\begin{aligned}
& \;\quad a_h(u_h-u_I,v_h)
\\ 
& = \sum_{K\in \mathcal  T_h} 
\bigl(f,\Pi_K v_h\bigr)_K - a_h(u_I,v_h)
=\sum_{K\in \mathcal  T_h} 
\bigl(-\Delta u,\Pi_K v_h\bigr)_K - a_h(u_I,v_h)
\\
&= \sum_{K\in \mathcal  T_h} 
\bigl(\nabla \Pi_K u, \nabla \Pi_K v_h\bigr)_K
- \sum_{K\in \mathcal  T_h} 
\left\langle \nabla u \cdot \bm{n}, 
\Pi_K v_h\right\rangle_{\partial K} - a_h(u_I,v_h)
\\
& = \sum_{K\in \mathcal  T_h} 
\bigl(\nabla \Pi_K (u-u_I), \nabla \Pi_K v_h\bigr)_K
+ \sum_{K\in \mathcal  T_h} \sum_{F\in \mathcal{F}_h(K)}
\left\langle \nabla u \cdot \bm{n}, 
Q_F v_h - \Pi_K v_h\right\rangle_{F}
\\
& \quad -\sum_{K\in \mathcal  T_h} S_{K}\bigl(u_I- \Pi_{\omega_K} u_I, 
v_h- \Pi_{\omega_K} v_h\bigr).
\end{aligned}
\end{equation}
\revision{We note that in the derivation above, on each face $F$, $Q_F v_h$ which is 
single-valued on $F$ can be freely inserted into boundary integrals since the 
inter-element jump of $\nabla u\cdot \bm{n}$ on $F$ vanishes by the assumption that 
$f\in L^2(\Omega)$.} 

Moreover, since $\chi_F(u-u_I) = 0$ for all faces $F$ by 
\eqref{eq:intp-loc}, by Lemma 
\ref{lemma:pr-id} we have $\Pi_K (u-u_I)= 0$, the first 
term in \eqref{eq:err-intbyparts} vanishes. Lastly, using the fact that in the 
lowest order case, since 
$u_{\pi}\in\mathbb{P}_1(K)$, $\Delta u_{\pi} = 0$, the following zero term can be 
inserted into the boundary integral in \eqref{eq:err-intbyparts} to get 
\eqref{eq:error}:
\[
\begin{aligned}
&\sum_{F\in \partial K} \left\langle \nabla u_{\pi} \cdot \bm{n}, 
Q_F v_h - \Pi_K v_h\right\rangle_{F}
= \left\langle \nabla u_{\pi} \cdot \bm{n},  
v_h - \Pi_K v_h\right\rangle_{\partial K}
\\
=& \,\bigl(\Delta u_{\pi},  v_h - \Pi_K 
v_h\bigr)_K+ \bigl(\nabla u_{\pi}, \nabla (v_h - \Pi_K 
v_h)\bigr)_K=0.
\end{aligned}
\]
\end{proof}

\begin{lemma}[An a priori error estimate on isotropic meshes] 
\label{lemma:err-estimate-iso}
Under the same setting with Lemma \ref{lemma:error-omega}, when the mesh 
$\mathcal{T}_h$ satisfies Assumptions \hyperref[asp:B]{\bf A--B}, it 
holds that
\begin{equation}
\label{eq:err-estimate-iso}
\begin{aligned}
\tnorm{u_h-u_I}^2 \lesssim  & \; \sum_{K\in \mathcal T_h}\sum_{F\in 
\mathcal{F}_h(K)}
h_{K}\norm{\nabla\big(u - \Pi_{K} u \big)\cdot \bm{n}}_{0,F}^2\\
 \;+ &   
\sum_{K\in \mathcal T_h}\sum_{F\in \mathcal{F}_h(K)}
h_{K}^{-1} \norm{u-\Pi_{K} u}_{0,F}^2
\end{aligned}
\end{equation} 
\end{lemma}
\begin{proof}
As $\mathcal{T}_h$ contains only isotropic elements, $\omega_K = K$ for all $K\in 
\mathcal{T}_h$. Let $u_{\pi} = \Pi_K u$ and $v_h = u_h - u_I$ in 
\eqref{eq:error}. Similarly to the proof of Lemma 
\ref{lemma:error-omega}, definition \eqref{eq:dofs-face} of DoFs with 
\eqref{eq:intp-loc} implies that $Q_F u = Q_F u_I$, hence $\Pi_{K} u_I = 
\Pi_{K} u$ by Lemma \ref{lemma:pr-id}. As a result, the stabilization term in 
\eqref{eq:error} can be estimated as follows:
\begin{equation}
\label{eq:err-sta}
\begin{aligned}
& \quad S_{K}\bigl(u_I- \Pi_{K} u_I, v_h- \Pi_{K} v_h\bigr)
\\
\leq &
\sum_{F\in \mathcal F_h(K)}h_K^{-1} \norm{Q_F (u_I-\Pi_{K} u_I)}_{0,F}
\norm{Q_F (v_h-\Pi_{K} v_h)}_{0,F}
\\
\leq &
\left(\sum_{F\in \mathcal F_h(K)}h_K^{-1} \norm{Q_F (u-\Pi_{K} 
u)}_{0,F}^2\right)^{1/2}
\left(\sum_{F\in \mathcal F_h(K)}h_K^{-1} \norm{Q_F (v_h-\Pi_{K} 
v_h)}_{0,F}\right)^{1/2},
\end{aligned}
\end{equation}
in which the first term can be estimated by $\norm{Q_F (u-\Pi_{K} 
u)}_{0,F} \leq \norm{ u-\Pi_{K} u}_{0,F}$, and the second term is a part of 
$\tnorm{v_h}$.
For the boundary integral term in \eqref{eq:error}, after using 
the Cauchy-Schwarz inequality on each face $F$, 
\begin{equation}
\label{eq:err-bd}
\left\langle \nabla (u -\Pi_K u)\cdot \bm{n}, 
Q_F v_h - \Pi_K v_h\right\rangle_{F}\leq 
\norm{\nabla (u -\Pi_K u )\cdot \bm{n}}_{0,F}
\norm{Q_F  v_h - \Pi_{K} v_h}_{0,F},
\end{equation}
we assign $h_K^{1/2}$ to the first term and $h_K^{-1/2}$ to the second term in 
\eqref{eq:err-bd}, and apply the triangle inequality as follows:
\begin{equation}
\label{eq:err-bd1}
\norm{Q_F  v_h - \Pi_{K} v_h}_{0,F}
\leq \norm{Q_F  (v_h -  \Pi_{K} v_h)}_{0,F}
+\norm{Q_F  \Pi_{K} v_h - \Pi_{K} v_h}_{0,F}.
\end{equation}
Consequently, the first term above, together with the weight $h_{K}^{-1/2}$, is now 
a part of $\tnorm{v_h}$. Applying the Poincar\'{e} inequality for the linear 
polynomial $\Pi_{K} v_h$ on face $F$ in Lemma \ref{lemma:poincare-linear-face} on 
the second term above, together with $|F|h_F\lesssim |F|l_F \leq |K|$ 
implied by the height condition \hyperref[asp:B]{\bf B}, leads to:
\begin{equation}
\label{eq:err-bdF}
h_{K}^{-1/2}\norm{Q_F  \Pi_{K} v_h - \Pi_{K} v_h}_{0,F} 
\leq h_F^{1/2}\norm{\nabla_F \Pi_{K} v_h}_{0,F} 
\lesssim \;\norm{\nabla \Pi_{K} v_h}_{0,K},
\end{equation}
which is a part of $\tnorm{v_h}$. Lastly summing up \eqref{eq:err-bd} in 
$\ell^2$-sense yields the lemma.
\end{proof}

With the a priori error estimate in Lemma \ref{lemma:err-estimate-iso}, it suffices 
to estimate the two terms from estimate \eqref{eq:err-estimate-iso}. First we 
estimate $(u - \Pi_{K}u)$ in the following lemma.

\begin{lemma}[\revision{Error estimate of $\Pi_{K}$ on an isotropic element}]
\label{lemma:err-prK}
When $K$ satisfies Assumptions \hyperref[asp:A]{\bf A--B--C}, for $u\in 
H^2\big(\operatorname{conv}(K)\big)$ it holds that:
\begin{equation}
\label{eq:err-prK}
h_{K}^{-1}\norm{u - \Pi_{K}u }_{0,K} +\norm{\nabla (u - \Pi_{K}u )}_{0,K} \lesssim 
h_{K} |u|_{2, \operatorname{conv}(K)}. 
\end{equation}
\end{lemma}
\begin{proof}
Since $\nabla \Pi_{K}u = \overline{\nabla u}^{K}$, the estimate in the second term 
follows from the Poincar\'e inequality in Lemma 
\ref{lemma:poincare-convex}. For the first term, by constraint 
\eqref{eq:pr-constraint}, applying 
the Poincar\'{e} inequality in Lemma \ref{lemma:poincare-b} and the triangle 
inequality lead to:
\begin{equation}
\begin{aligned}
& h_{K}^{-1}\norm{u - \Pi_{K}u }_{0,K} \lesssim \norm{\nabla(u - \Pi_{K}u) 
}_{0,\operatorname{conv}(K)} 
\\
\leq &\; 
\bigl\| \nabla u - \overline{\nabla u}^{\operatorname{conv}(K)} 
\bigr\|_{0,\operatorname{conv}(K)} 
+ \bigl\|\overline{\nabla u}^{\operatorname{conv}(K)} - \overline{\nabla 
u}^{K} \bigr\|_{0,\operatorname{conv}(K)}.
\end{aligned}
\end{equation} 
For the second term above, Cauchy-Schwarz inequality, 
$|\operatorname{conv}(K)|\lesssim h_K^d$, and $|K|\eqsim h_K^d$ in Lemma 
\ref{lemma:iso} imply that
\begin{equation}
\begin{aligned}
&\bigl\|\overline{\nabla u}^{\operatorname{conv}(K)} - \overline{\nabla 
u}^{K} \bigr\|_{0,\operatorname{conv}(K)}
= |\operatorname{conv}(K)|^{1/2}\left|\frac{1}{|K|} \int_K \Big( {\nabla u} - 
\overline{\nabla 
u}^{\operatorname{conv}(K)} \Big) \right|
\\
\leq &\; \frac{|\operatorname{conv}(K)|^{1/2}}{|K|^{1/2}}
\bigl\| \nabla u - \overline{\nabla u}^{\operatorname{conv}(K)} 
\bigr\|_{0,K} \lesssim 
\bigl\| \nabla u - \overline{\nabla u}^{\operatorname{conv}(K)} 
\bigr\|_{0,\operatorname{conv}(K)}.
\end{aligned}
\end{equation}
Consequently, the desired estimate follows from applying Lemma 
\ref{lemma:poincare-convex} on $\operatorname{conv}(K)$ and the fact that the 
diameter of $\operatorname{conv}(K)$ is $h_K$.
\end{proof}

\begin{lemma}[\revision{Error estimate of the normal derivative of $\Pi_{K}$}]
\label{lemma:err-prnd}
For $K\in \mathcal{T}_h$, provided that every $F$ 
satisfies Assumption \hyperref[asp:B]{\bf B--C}, the following 
error estimate holds on a face $F\in \mathcal{F}_h(K)$ for $u\in 
H^2\bigl(\operatorname{conv}(K)\bigr)$
\begin{equation}
h_{K}^{1/2}\norm{\nabla (u - \Pi_{K}  u )\cdot \bm{n}}_{0,F}
\lesssim h_{K} |u|_{2,\operatorname{conv}(K)}.
\end{equation}
\end{lemma}
\begin{proof}
By Assumption \hyperref[asp:B]{\bf B--C}, we apply trace 
inequality \eqref{eq:tr-hK} toward $K$
\begin{equation}
\label{eq:err-trF}
h_{K}^{1/2}\norm{\nabla (u - \Pi_{K}  u )\cdot \bm{n}}_{0,F} 
\lesssim \norm{\nabla (u - \Pi_{K}  u ) }_{0,K} 
+ h_{K} |{u}|_{2,K}.
\end{equation}
The lemma then follows from Lemma \ref{lemma:err-prK}. 
\end{proof}

\begin{lemma}[Error estimate of $\Pi_{K}$ on a face]
\label{lemma:err-pr}
For $K\in \mathcal{T}_h$, provided that $K$ 
satisfies Assumptions \hyperref[asp:B]{\bf A--B--C}, the 
following error estimate holds on a face $F\in \mathcal{F}_h(K)$ for $u\in 
H^2\bigl(\operatorname{conv}(K)\bigr)$:
\begin{equation}
\label{eq:err-pr-iso}
h_{K}^{-1/2}\norm{u - \Pi_{K}u }_{0,F} \lesssim h_{K} 
|u|_{2,\operatorname{conv}(K)}.
\end{equation}
\end{lemma}
\begin{proof}
Since for every $F\in\mathcal{F}_h(K)$, $F$ satisfies Assumptions 
\hyperref[asp:B]{\bf B--C} 
with respect to $K$, by the trace inequality in Lemma \ref{lemma:trace-hK}, we have:
\begin{equation}
\label{eq:err-trFL2}
h_{K}^{-1/2}\norm{u - \Pi_{K}u }_{0,F} \lesssim 
h_{K}^{-1}\norm{u - \Pi_{K}u }_{0,K} 
+ \norm{\nabla (u - \Pi_{K}u) }_{0,K},
\end{equation}
which yields the desired estimate by Lemma \ref{lemma:err-prK}.
\end{proof}

Now the a priori convergence result for the lowest order nonconforming 
VEM on an isotropic mesh can be summarized as follows.
\begin{theorem}[\revision{Convergence on isotropic meshes}]
\label{theorem:convergence-isotropic}
Assume that the mesh $\mathcal{T}_h$ is isotropic in the sense of Assumptions 
\hyperref[asp:A]{\bf A--B--C--E}. When the 
solution $u$ to \eqref{eq:model-weak} satisfies $u\in H^2(\Omega)$, the 
following error estimate holds for the solution $u_h$ to \eqref{eq:VEM}:
\begin{equation}
\label{eq:convergence-isotropic} 
\tnorm{u-u_h} \lesssim h \| u \|_{2,\Omega}.
\end{equation}
\end{theorem}

\begin{proof}
First of all, we apply the Stein's extension theorem (\cite{Stein:1970Singular} 
Theorem 6.5) to $u\in H^2(\Omega)$ to get a 
function $u_E \in H^2(\mathbb R^d), u_E |_{\Omega} = u|_{\Omega}$, and $\| u_E 
\|_{2, \mathbb R^d}\leq C(\Omega) \| u \|_{2,\Omega}$. With this extension $u_E\in 
H^2\big(\operatorname{conv}(K)\big)$ for any $K\in \mathcal T_h$. 

Secondly, the estimates from Lemma \ref{lemma:err-prnd} and \ref{lemma:err-pr} are 
plugged into Lemma \ref{lemma:err-estimate-iso}, and Assumption \hyperref[asp:A]{\bf 
A} ensures that these estimates are summed up bounded times on a fixed 
element. Meanwhile, Assumption \hyperref[asp:E]{\bf E} implies that the integral on 
the overlap $\operatorname{conv}(K)\cap \operatorname{conv}({K'})$ is repeated 
bounded times for neighboring $K, 
K'\in \mathcal{T}_h$. Therefore,
\begin{equation}
\tnorm{u_I-u_h}^2 \lesssim \sum_{K\in \mathcal T_h}h_{K}^2 
|u_E|_{2,\operatorname{conv}(K)}^2 
\lesssim h^2|u_E|_{2,\operatorname{conv}(\Omega)}^2 \lesssim h^2\|u\|_{2,\Omega}^2.
\end{equation}
As $\tnorm{u-u_I}=0$ by the construction of $u_I$ and \eqref{eq:norm-ah}, the 
theorem follows.
\end{proof}

\subsection{\revision{A priori error estimates on anisotropic elements}}
\revision{In the vanilla error equation \eqref{eq:error}, a boundary term that 
involves $\Pi_K v_h$ is present. For an anisotropic element $K$, key estimates 
including \eqref{eq:err-bdF}, \eqref{eq:err-trF}, and \eqref{eq:err-trFL2} will 
become 
problematic where the Assumptions \hyperref[asp:B]{\bf B--C} are violated.} Instead, 
the boundary term will be lifted to its 
isotropic extended element patch $\omega_K$, and thus in next lemma we aim to 
replace $\Pi_K v$ by $\Pi_{\omega_K} v$ in the error equation \eqref{eq:error} 
taking the anisotropic elements into account.

\begin{lemma}[Expanded error equation]
\label{lemma:error-exp}
Under the same setting with Lemma \ref{lemma:error-omega}, it holds that
\begin{equation}
\label{eq:error2}
\begin{aligned}
a_h(u_h-u_I,v_h)  =& \sum_{K\in \mathcal{T}_h}\sum_{F\in \mathcal{F}_h(K)} 
\left\langle \nabla (u -u_{\pi} )\cdot \bm{n}, 
Q_F v_h - \Pi_{\omega_K} v_h\right\rangle_{F}
\\
\quad & -  \sum_{K\in \mathcal{T}_h} 
S_{K}\bigl(u_I-\Pi_{\omega_K} u_I, v_h - \Pi_{\omega_K} v_h\bigr)\\
\quad &+ \sum_{K\in \mathcal{T}_h}\left(  \nabla (u -u_{\pi} ),  \nabla 
\bigl( \Pi_{\omega_K} v_h -\Pi_{K}  v_h \bigr)\right)_K \\
\quad & - \sum_{K\in \mathcal{T}_h}(f, \Pi_{\omega_K} v_h -\Pi_{K} v_h )_K.
\end{aligned}
\end{equation}
\end{lemma}
\begin{proof}
Starting with \eqref{eq:error}, we only need to expand the difference term as follows
\begin{align*}
&\sum_{F\in \mathcal{F}_h(K)} 
\left\langle \nabla (u -u_{\pi} )\cdot \bm{n}, 
\Pi_{\omega_K} v_h- \Pi_K  v_h\right\rangle_{F}\\
=& \,\bigl(\Delta (u - u_{\pi}),  \Pi_{\omega_K} v_h- \Pi_K  v_h \bigr)_K+ 
\bigl(\nabla (u - u_{\pi}), \nabla (\Pi_{\omega_K} v_h- \Pi_K  v_h)\bigr)_K\\
= & - (f, \Pi_{\omega_K} v_h -\Pi_{K} v_h )_K + \bigl(\nabla (u - u_{\pi}), 
\nabla (\Pi_{\omega_K} v_h- \Pi_K  v_h)\bigr)_K.
\end{align*}
\end{proof}

For the last term in \eqref{eq:error2} involving difference in an $L^2$-inner 
product, Poincar\'e inequalities with appropriate constraints can be applied to 
change it to the energy norm.
\begin{lemma}[Difference between projections] 
\label{lemma:err-prdiff}
If Assumption \hyperref[asp:D]{\bf D} is met for $K$, denote 
$\tnorm{v_h}_{\omega_K}^2 
:=\sum_{K\in \mathcal{T}_h(\omega_K)} \tnorm{v_h}_{K}^2$, then
\begin{equation}
\norm{\Pi_{K}  v_h - \Pi_{\omega_K} v_h}_{0,K} \lesssim h_{\omega_K} 
\tnorm{v_h}_{\omega_K}.
\end{equation}
\end{lemma}
\begin{proof}
As we choose the constraint $\int_{\partial \omega_K} \Pi_{K}  v_h = 
\int_{\partial  \omega_K} \Pi_{\omega_K}  v_h = \int_{\partial \omega_K}v_h$, and 
$\Pi_{K}  v_h - \Pi_{\omega_K} v_h$ is a linear polynomial on $K$, the 
estimate is a direct consequence of Lemma \ref{lemma:poincare-linear-patch} and 
 \ref{lemma:pr-split}.
\end{proof}

\begin{lemma}[A priori error estimate using the expanded error equation] 
\label{lemma:err-estimate}
Under the same setting with Lemma \ref{lemma:error-omega}, when $\mathcal{T}_h$ 
satisfies Assumptions \hyperref[asp:A]{\bf A--D}, it holds that
\begin{equation}
\label{eq:err-estimate}
\begin{aligned}
\tnorm{u_h-u_I}^2 \lesssim  & \;\sum_{K\in \mathcal T_h}\sum_{F\in \mathcal{F}_h(K)}
h_{\omega_K}\norm{\nabla\big(u - \Pi_{\omega_K} u \big)\cdot \bm{n}}_{0,F}^2\\
 \;+ &   
\sum_{K\in \mathcal T_h}\sum_{F\in \mathcal{F}_h(K)}
h_{\omega_K}^{-1} \norm{u-\Pi_{\omega_K} u}_{0,F}^2
\\
\; + & \sum_{K\in \mathcal T_h } 
\norm{\nabla(u-\Pi_{\omega_K} u)}_{0,K}^2+ \sum_{K\in \mathcal T_h} 
h_{\omega_K}^2\norm{f}_{0,K}^2 
\end{aligned}
\end{equation} 
\end{lemma}
\begin{proof}
We proceed similarly with the proof of Lemma \ref{lemma:err-estimate-iso} by 
choosing $v_h = u_h - u_I$, yet letting $u_{\pi} = \Pi_{\omega_K}u$ instead in 
\eqref{eq:error2}. The four terms in \eqref{eq:error2} shall be estimated in a 
backward order. For the fourth term, by the Cauchy-Schwarz inequality and applying 
Lemma \ref{lemma:err-prdiff}, we have
\begin{equation}
\big(f, (\Pi_{K} - \Pi_{\omega_K} )v_h \big)_K \leq \norm{f}_{0,K}\| 
(\Pi_{K} - \Pi_{\omega_K} )v_h\|_{K}\lesssim  h_{\omega_K} \norm{f}_{0,K} 
\tnorm{v_h}_{\omega_K}.
\end{equation}
The third term can be estimated in a similar fashion by applying the 
Cauchy-Schwarz inequality, and applying Lemma \ref{lemma:pr-split} 
to get  
\begin{equation}
\norm{\nabla (\Pi_{\omega_K} v_h -\Pi_{K}  v_h)}_{0,K} \leq 
\norm{\nabla \Pi_{\omega_K} v_h }_{0,\omega_K} + \norm{\nabla \Pi_{K} v_h }_{0,K}
\lesssim \tnorm{v_h}_{\omega_K}.
\end{equation} 
For the second term which is the stabilization, a similar argument with 
\eqref{eq:err-sta} in the proof of Lemma \ref{lemma:err-estimate-iso} can be used. 
By $\mathcal{F}_h(\omega_K)\subset \mathcal{F}_h$, Lemma 
\ref{lemma:pr-id} implies that $\Pi_{\omega_K} u_I = \Pi_{\omega_K} u$, which leads 
a similar estimate as the second term in \eqref{eq:err-estimate-iso}, and 
the difference is that $\Pi_K$ and $h_K$ are replaced in \eqref{eq:err-sta} by 
$\Pi_{\omega_K}$ and $h_{\omega_K}$, respectively.

The first term of \eqref{eq:error2} is treated similarly with 
\eqref{eq:err-bd} and \eqref{eq:err-bd1}, then since $\omega_K$ is isotropic, the 
rest of the proof, in which  $\Pi_K$ and $h_K$ are replaced by $\Pi_{\omega_K}$ and 
$h_{\omega_K}$, proceeds exactly the same with \eqref{eq:err-bdF}:
\begin{equation}
\label{eq:error-3}
h_{\omega_K}^{-1/2}\norm{Q_F  \Pi_{\omega_K} v_h - \Pi_{\omega_K} v_h}_{0,F} 
\lesssim h_F^{1/2}\norm{\nabla_F \Pi_{\omega_K} v_h}_{0,F} 
\lesssim \;\norm{\nabla \Pi_{\omega_K} v_h}_{0,\omega_K},
\end{equation}
and finally the lemma follows from Lemma \ref{lemma:pr-split}.
\end{proof}

With the a priori error estimate in Lemma \ref{lemma:err-estimate}, it suffices to 
estimate term by term in \eqref{eq:err-estimate}. Since now it involves  
only the error of the projection on the isotropic extended patch $\omega_K$, the 
estimates in Lemma \ref{lemma:err-prK}, \ref{lemma:err-prnd}, and \ref{lemma:err-pr} 
can be reused by replacing the $K$ with $\omega_K$, both of which are isotropic.

The next theorem summarizes an a priori convergence result that incorporates 
possible anisotropic elements (cf. Theore \ref{theorem:convergence-isotropic}), 
\revision{and we remark that 
Assumption \hyperref[asp:D]{\bf D} includes the scenarios when Assumptions 
\hyperref[asp:B]{\bf B--C} are met as $\omega_K = K$}. 
\begin{theorem}[Convergence on possible anisotropic meshes]
\label{theorem:convergence}
Assume that the mesh $\mathcal{T}_h$ satisfies Assumptions \hyperref[asp:D]{\bf 
A--D--E}. When the 
solution $u$ to problem \eqref{eq:model-weak} satisfies $u\in H^2(\Omega)$, the 
following error estimate holds for the solution $u_h$ to problem \eqref{eq:VEM}:
\begin{equation}
\label{eq:convergence} 
\tnorm{u-u_h} \lesssim h \| u \|_{2,\Omega}.
\end{equation}
\end{theorem}
\begin{proof} We proceed exactly like Theorem 
\ref{theorem:convergence-isotropic} by extending $u$ to $H^2(\mathbb{R}^d)$ first. 
The estimate in Lemma \ref{lemma:err-prK} can be changed straightforwardly on 
$\omega_K$:
\begin{equation}
\norm{\nabla (u - \Pi_{\omega_K}u )}_{0,\omega_K} \lesssim h_{\omega_K} |u|_{2, 
\operatorname{conv}(\omega_K)}.
\end{equation}
Since $\omega_K$ satisfies Assumptions 
\hyperref[asp:B]{\bf B--C} by Assumption \hyperref[asp:D]{\bf D}, the estimates in 
Lemma \ref{lemma:err-pr} and \ref{lemma:err-prnd} are changed accordingly on 
$\omega_K$ as well:
\begin{align}
& h_{\omega_K}^{-1/2}\norm{u - \Pi_{\omega_K}u }_{0,F} \lesssim h_{\omega_K} 
|u|_{2,\operatorname{conv}(\omega_K)},
\\
\text{ and } & h_{\omega_K}^{1/2}\norm{\nabla \bigl(u - \Pi_{\omega_K}  u 
\bigr)\cdot 
\bm{n}}_{0,F}
\lesssim h_{\omega_K} |u|_{2,\operatorname{conv}(\omega_K)}.
\end{align}
After these estimates are plugged 
into Lemma \ref{lemma:err-estimate}, Assumptions \hyperref[asp:E]{\bf A--E} are 
applied in the same way with Theorem \ref{theorem:convergence-isotropic}, except now 
we consider the integral overlap on patches 
$\operatorname{conv}(\omega_K)\cap \operatorname{conv}(\omega_{K'})$ for neighboring 
elements. Upon using the 
fact that $\norm{f}_{0,K} = \norm{\Delta u}_{0,K} \leq |u|_{2,K}$, we obtain 
\begin{equation}
\tnorm{u_I-u_h}^2 \lesssim \sum_{K\in \mathcal T_h}h_{\omega_K}^2 
|u_E|_{2,\operatorname{conv}(\omega_K)}^2 
\lesssim h^2|u_E|_{2,\operatorname{conv}(\Omega)}^2 \lesssim h^2\|u\|_{2,\Omega}^2,
\end{equation}
and the rest of the proof is the same with the one in Theorem 
\ref{theorem:convergence-isotropic}.
\end{proof}

\section{Concluding remarks and future study}
\label{sec:remark}
The error analysis in this paper further relaxes and extends to 3-D of the geometry 
constraints for the linear VEM's conforming counterpart 
in~\cite{Cao;Chen:2018Anisotropic}, and Assumptions 
\hyperref[asp:B]{\bf B--C} can generalized for arbitrary dimension. Since the 
stabilization is of a weighted $L^2$-type, unlike the analysis in 
\cite{Cao;Chen:2018Anisotropic} bridging the stabilization with a discrete 
$H^{1/2}$-norm on boundary, the versatility of the VEM framework allows the 
stabilization in the nonconforming VEM to be more flexible and localizable. 

As a result, even for the isotropic case in 3D, the current analysis allows a tiny 
face and anisotropic face provided that the element is isotropic in the sense of 
Assumptions \hyperref[asp:B]{\bf B--C}, in addition to two alternative shape 
regularity conditions in Assumptions \hyperref[asp:A]{\bf A--E}. \revision{In our 
view, being ``isotropic'' for an element is a localized property near a face, in 
that the tangential direction and the normal direction of this face are comparable 
by the height condition \hyperref[asp:B]{\bf B}. Furthermore, the hourglass 
condition \hyperref[asp:C]{{\bf C}} can be viewed as a localized star-shaped 
condition for face $F$, where $K_F$ can be different for different $F$. The 
convexity of $K_F$ allows that any line connecting a point in $K_F$ to a face of 
$P_F$ is entirely in $K_F$ (cf. Remark \ref{remark:trace}), which makes $K_F$'s 
role similar to the inscribed ball to which $K$ is uniformly star-shaped in the 
traditional VEM analysis}. Meanwhile 
the existence of concave faces are allowed in the decomposition sense. 

One of the major factors facilitating the new analysis is the introduction of the 
stabilization on an extended element patch in the discretization 
\eqref{eq:bilinear-aniso}. We remark some of the concerns regarding 
the implementation using a 2-D example in the following subsection.
\subsection{Implementation remarks on the extended patch}
\label{sec:implementation}
When $\mathcal{T}_h$ is a body-fitted mesh generated by 
cutting a shape-regular background grid, $\omega_K$, which is only needed for 
certain anisotropic cut elements $\{K\}$, can be naturally chosen as the patch 
joining $K$ with one of $K$'s nearest neighbors in the background mesh. Here we 
shall illustrate using the elements $K_1,K_2\in \mathcal{T}_h$ in Figure 
\ref{fig:cutelem}, which are cut from a Cartesian 
mesh in 2-D.

When $\mathcal{T}_h$ is not generated from cutting a background shape-regular mesh, 
the situation is much more complicated, as the search for a possible extended 
element patch may produce more overhead. To pin down $\omega_K$ for an anisotropic 
$K$, one possible procedure is to 
estimate the chunkiness parameter of $K$ first: computing $|K|$ and diameter $h_F$ 
of an edge or a face $F\subset \partial K$, if the ratio $h_F/|K|^{1/d}$ is bigger 
than a threshold, then $K$ 
shall be treated as anisotropic. Starting from an anisotropic element, we can 
join its immediate neighbor sharing an edge or a face with $K$ to form $\omega_K$ 
and having the minimum chunkiness parameter among all neighbors. 
Lastly this test is repeated when necessary until $\omega_K$ passes the test. 

In the implementation, using the data structure for polyhedral 
elements~\cite{Chen;Wei;Wen:2017interface-fitted}, we can use one array to store all 
faces and another to story the indices of the polyhedra to which 
the every face belongs. In this regard, the elements are represented 
by these two arrays, and the merging of neighboring elements is very efficient; we 
refer the reader to~\cite[Section 3.2]{Chen;Wei;Wen:2017interface-fitted} for 
technical details.

\subsection{\revision{Implementation of the new stabilization}}
\label{sec:implementstablization}
In the element-wise assembling of the matrix corresponding to the bilinear form 
\eqref{eq:bilinear-aniso}, we shall separate the terms 
of the projected gradient part $\bigl(\nabla \Pi_K u_h, \nabla \Pi_K v_h\bigr)_K$ 
and the stabilization part $S_{K}\bigl(u_h-\Pi_{\omega_K} u_h,v_h- \Pi_{\omega_K} 
v_h\bigr)$. The former remains unchanged from the unmodified formulation. We focus 
on the implementation of the stabilization term.

For each anisotropic element $K$, assume that we have found an extended patch 
$\omega_K$ which itself is also represented as a polytopal element. Then 
${\Pi}_{\omega_K}$ 
can be realized by a matrix $\bm{\Pi}_{\omega_K}$ of size $(d+1)\times 
n_{\omega_K}$. The $L^2$-projection $Q_F$ on $F\in \mathcal{F}_h(\partial K)$ 
applied to a linear 
polynomial is realized by the DoF matrix $\mathbf{D}$ of size $n_K \times (d+1)$. 
See 
\cite{Beirao-da-Veiga;Brezzi;Marini;Russo:2014hitchhikers} for detailed formulations 
of matrices $\bs \Pi_{\omega_K}$ and $\mathbf{D}$. Denote by $\bar{\mathbf{I}} = 
(\mathbf{I} \quad \bs 0)_{n_K \times n_{\omega_K}}$ the extended identity matrix. 
The stabilization on 
$K$ can be realized by an $n_{\omega_K}\times n_{\omega_K}$ local matrix
\begin{equation}\label{eq:matrixSK}
h_{\omega_K}^{-1} (\bar{\mathbf{I}} - \mathbf{D} \bs \Pi_{\omega_K})^{\intercal}
{\rm diag}(|F_1|, \ldots, |F_{n_K}|)(\bar{\mathbf{I}} - \mathbf{D}\bs 
\Pi_{\omega_K}).
\end{equation}
Thus the standard assembling procedure looping over all elements can be applied to 
assemble a global one. 

As a comparison, the original stabilization using $\Pi_K$ is a matrix of size 
$n_{K}\times n_{K}$ and in the form $h_{K}^{-1} (\mathbf{I} - \mathbf{D}\bs 
\Pi_{K})^{\intercal}{\rm diag}(|F_1|, \ldots, |F_{n_K}|)(\mathbf{I} - \mathbf{D}\bs 
\Pi_{K})$. We note that one effect 
of enlarging the element is that the stabilization matrix \eqref{eq:matrixSK} is 
denser or equivalently the stencil is larger.

\section*{Acknowledgments}
We appreciate an anonymous reviewer for bringing up several insightful 
questions which improved an early version of the paper.

% % % % % % % % % % % % % % % % % % % % % %
\bibliographystyle{siam}

\end{document}